\documentclass[12pt]{article}
\usepackage{amsmath}
\usepackage{amsfonts,amssymb}
\usepackage{theorem}
\usepackage[dvips]{graphics}
\usepackage{graphicx}

\usepackage{pstricks}
\usepackage{psfrag}
\usepackage[all]{xy}

\numberwithin{equation}{section}

\newcommand{\bfzero}{{\bf 0}}

\newcommand{\OC}{{\mathcal{O C}}}

\newcommand{\Conf}{{\rm C o n f}}

\newcommand{\Coder}{{\rm Coder}}

\newcommand{\id}{{\rm id}}

\renewcommand{\dim}{{\rm d i m}\,}

\newcommand{\hpsi}{\hat{\psi}}

\newcommand{\hG}{\hat{G}}

\newcommand{\End}{{\rm E n d}}

\newcommand{\Spec}{{\rm S p e c}}

\newcommand{\bC}{\overline{C}}

\newcommand{\FME}{{\rm  F M E}}

\newcommand{\Cobar}{{\rm Cobar}}

\newcommand{\Sh}{{\rm Sh}}

\newcommand{\hQ}{\hat{Q}}
\newcommand{\hU}{\hat{U}}
\newcommand{\hF}{\hat{F}}
\newcommand{\hT}{\hat{T}}

\newcommand{\Cbu}{C^{\bullet}}

\newcommand{\Der}{{\rm Der}\,}
\newcommand{\Hom}{{\rm Hom}\,}

\newcommand{\sgn}{{\rm s g n}}

\newcommand{\Lie}{{\bf Lie}}

\newcommand{\cocomm}{{\bf cocomm}}
\newcommand{\coass}{{\bf coassoc}}
\newcommand{\Ger}{{\bf Ger}}

\newcommand{\SC}{{\rm SC}}

\newcommand{\bS}{{\bf S}}
\newcommand{\bH}{{\bf H}}
\renewcommand{\sc}{{\bf sc}}

\newcommand{\bul}{{\bullet}}

\newcommand{\la}{{\lambda}}

\newcommand{\si}{{\sigma}}
\newcommand{\ga}{{\gamma}}
\newcommand{\vf}{{\varphi}}
\newcommand{\ve}{{\varepsilon}}

\newcommand{\vr}{{\varrho}}

\newcommand{\ma}{{\mathfrak{o}}}
\newcommand{\mc}{{\mathfrak{c}}}

\renewcommand{\mp}{{\mathfrak{p}}}

\newcommand{\bs}{{\bf s}}

\newcommand{\pa}{{\partial}}

\newcommand{\cD}{{\cal D}}

\newcommand{\cH}{{\cal H}}

\newcommand{\cL}{{\cal L}}
\newcommand{\cC}{{\cal C}}

\newcommand{\cV}{{\cal V}}
\newcommand{\cW}{{\cal W}}
\newcommand{\cO}{{\cal O}}

\newcommand{\bbQ}{{\mathbb Q}}
\newcommand{\bbC}{{\mathbb C}}
\newcommand{\bbR}{{\mathbb R}}

\newcommand{\bbK}{{\mathbb K}}

\newcommand{\La}{{\Lambda}}

\newcommand{\te}{\theta}

\newcommand{\de}{{\delta}}
\newcommand{\D}{{\Delta}}

\newcommand{\tS}{{\underline{S}}}
\newcommand{\tT}{{\underline{T}}}

\newcommand{\tcL}{{\widetilde{\mathcal L}}}
\newcommand{\tU}{{\widetilde{U}}}

\newcommand{\tcO}{{\widetilde{{\mathcal O}}}}

\newcommand{\brarrow}{\succ\rightarrow}

\date{}

\newtheorem{defi}{Definition}[section]

\newtheorem{teo}[defi]{Theorem}
\newtheorem{cor}[defi]{Corollary}

\newtheorem{pred}[defi]{Proposition}

\theorembodyfont{\rm}
\newtheorem{example}[defi]{Example}
\newtheorem{remark}[defi]{Remark}
\newtheorem{cond}[defi]{Condition}

\makeatletter
\newcommand\qedsymbol{\hbox{$\Box$}}
\newcommand\qed{\relax\ifmmode\Box\else
  {\unskip\nobreak\hfil\penalty50\hskip1em\null\nobreak\hfil\qedsymbol
  \parfillskip=\z@\finalhyphendemerits=0\endgraf}\fi}

\newenvironment{proof}[1][{}]{\par\noindent{\bf Proof}{#1}. }{\qed}

\title{Formality theorem for Hochschild cochains via
transfer}

\author{Vasily Dolgushev}

\begin{document}

\maketitle

\begin{center}
{\it Belatedly to Simon Lyakhovich on the occasion of his 50th birthday.}
\end{center}

\begin{abstract}
We construct a 2-colored operad  $\Ger^+_{\infty}$ which, on the 
one hand, extends the operad
$\Ger_{\infty}$ governing homotopy Gerstenhaber algebras and, on 
the other hand, extends the 2-colored operad governing open-closed homotopy 
algebras (OCHA). We show that Tamarkin's
$\Ger_{\infty}$-structure on the Hochschild cochain complex
$\Cbu(A,A)$ of an $A_{\infty}$-algebra $A$
 extends naturally to a  $\Ger^+_{\infty}$-structure 
on the pair $(\Cbu(A,A), A)$\,.
We show that a formality quasi-isomorphism for the Hochschild
cochains of the polynomial algebra can be obtained via
transfer of this $\Ger^+_{\infty}$-structure to the cohomology
of the pair $(\Cbu(A,A), A)$\,. We show that $\Ger^+_{\infty}$ is 
a sub DG operad of the first sheet $E^1(\SC)$ of the homology spectral sequence 
for the Fulton-MacPherson version $\SC$ of Voronov's 
Swiss Cheese operad. Finally, we prove that the 
DG operads $\Ger^+_{\infty}$ and $E^1(\SC)$ are 
non-formal.  
\end{abstract}

~\\
{\it Keywords:} Operads, formality theorems, homotopy algebras.  \\
{\it 2010 MSC:} 18D50, 55U35.

\tableofcontents

\section{Introduction}
This paper is devoted to a natural $2$-colored extension 
$\Ger^+_{\infty}$ of the operad $\Ger_{\infty}$ governing homotopy
Gerstenhaber algebras. Informally a $\Ger^+_{\infty}$-algebra
is a blend of two ubiquitous algebraic structures: $A_{\infty}$-algebra and
homotopy Gerstenhaber algebra. The notion of $A_{\infty}$-algebra 
was invented by J. Stasheff \cite{Stasheff}. The notion of homotopy Gerstenhaber 
algebra has its origin 
in works \cite{Ger}, \cite{Ezra}, \cite{braces} of M. Gerstenhaber, 
E. Getzler, and T. Kadeishvili. This latter notion was developed 
and formalized in works \cite{BF},  \cite{GerVor}, \cite{GJ}, \cite{KS}, \cite{M-Smith}, 
\cite{Dima-Proof}, \cite{Sasha}.  
In this paper we describe interesting 
relations of the operad $\Ger^+_{\infty}$ to Kontsevich's formality theorem \cite{K}, 
to the operad $\OC_{\infty}$ governing open-closed homotopy algebras \cite{OCHAStasheff}, 
and to Voronov's Swiss Cheese operad \cite{Sasha-SC}.
 
We define the DG operad $\Ger^+_{\infty}$ as the cobar 
construction of a cooperad $\bS$ which is closely related 
to the cohomology cooperad $H^{\bul}(\SC)$ of Voronov's Swiss 
Cheese operad $\SC$. We show that algebras over  $\Ger^+_{\infty}$
are pairs $(V,A)$ with the following data: 
\begin{itemize}
\item A $\Ger_{\infty}$-structure on $V$\,, 
\item An $A_{\infty}$-structure on $A$\,, and
\item an $L_{\infty}$-morphism from $V$ to the Hochschild cochain 
complex $\Cbu(A,A)$ of $A$\,. 
\end{itemize}

Using this description we show that for every $A_{\infty}$-algebra $A$
Tamarkin's $\Ger_{\infty}$-structure on the Hochschild cochain complex 
$\Cbu(A,A)$ extends naturally to a  $\Ger^+_{\infty}$-structure on the 
pair $(\Cbu(A,A), A)$\,.

We show that the DG operad $\Ger^+_{\infty}$ satisfies the 
minimal model condition \cite{Sullivan}. This condition implies that  
the operad $\Ger^+_{\infty}$ is cofibrant and $\Ger^+_{\infty}$-structures 
can be transferred across quasi-isomorphisms. 

Using Tamarkin's rigidity (see Theorem \ref{thm:rigidity}) 
for the Hochschild cohomology $HH^{\bul}(A,A)$ of 
the polynomial algebra
$A =  \bbK[x^1, \dots, x^d]$ we show that a formality quasi-isomorphism 
for Hochschild cochains $\Cbu(A,A)$ coincides with 
the corresponding part of the transferred
$\Ger^+_{\infty}$-structure on the pair $(HH^{\bul}(A,A), A)$\,.

Furthermore, we show that the operad $\Ger^+_{\infty}$
allows us to establish a  
necessary and sufficient condition (Theorem \ref{Teorema})
for an abstract $L_{\infty}$ morphism 
$$
HH^{\bul}(A,A)  \brarrow \Cbu(A,A)
$$
to be homotopic to an $L_{\infty}$ quasi-isomorphism obtained 
via Tamarkin's 
procedure  \cite{DTT}, \cite{Hinich-pro-Dimu}, \cite{Dima-Proof}. 

We also show that $\Ger^+_{\infty}$ is a sub-operad of 
the first sheet $E^1(\SC)$ of the homology spectral sequence for the 
topological Swiss Cheese operad considered by A. Voronov 
in \cite{Sasha-SC}.  Finally we prove that the DG operads 
$\Ger^+_{\infty}$ and  $E^1(\SC)$ are non-formal.  

The paper is organized as follows.

In Section \ref{sect:prelim}  we fix notation and conventions. 
In this section we also  
recall necessary facts about Hochschild cochain complex 
of an $A_{\infty}$-algebra.
In Section \ref{sect:Ger-plus} we introduce the operad $\Ger^+_{\infty}$
and define $\Ger^+_{\infty}$-morphism as well as $\Ger^+_{\infty}$ 
quasi-isomorphisms. In Section \ref{sect:Ger-plus-alg} we give 
an alternative description of $\Ger^+_{\infty}$-algebras and show that 
the operad $\Ger^+_{\infty}$ is an extension of the operad $\OC_{\infty}$
governing open-closed homotopy algebras (OCHA). In this section we also
discuss 
homotopies  of  $\Ger^+_{\infty}$-algebras, prove the 
transfer theorem, and construct a $\Ger^+_{\infty}$-algebra 
structure on the pair $(\Cbu(A,A), A)$\,. 
In Section \ref{sect:via-transfer} we establish a necessary and 
sufficient condition for an abstract $L_{\infty}$-morphism
from $HH^{\bul}(A,A)$ to $\Cbu(A,A)$ to be homotopic to 
the one obtained via Tamarkin's procedure \cite{DTT}, \cite{Hinich-pro-Dimu},
\cite{Dima-Proof}.   In this section 
 we also show that a formality 
quasi-isomorphism for Hochschild cochains $\Cbu(A,A)$ of 
the polynomial algebra $A =  \bbK[x^1, \dots, x^d]$ 
is obtained via transfer of the  $\Ger^+_{\infty}$-algebra on   
 $(\Cbu(A,A), A)$ to the cohomology.
In Section \ref{sect:SC-Ger-plus} we describe the first sheet of the homology 
spectral sequence $E^1(\SC)$ for the Fulton-MacPherson version 
$\SC$ of Voronov's Swiss Cheese operad \cite{Sasha-SC}. 
In this section we show that $\Ger^+_{\infty}$ is a suboperad of 
the DG operad  $E^1(\SC)$\,. 
In Section  \ref{sect:non-formal} we prove that the DG operads 
$\Ger^+_{\infty}$ and $E^1(\SC)$ are non-formal. Finally, in 
the concluding section, we discuss a few open questions.

\subsection*{Acknowledgment}
I was secretly introduced to the fascinating world of homotopy
algebras when I studied quantization of gauge
systems with my advisor Simon Lyakhovich in Tomsk State University.
I would like to thank Simon for his guidance and
greet him with his anniversary.
I would like to thank S. Shadrin, J. Stasheff, B. Vallette, A. Voronov and
T. Willwacher for useful discussions. I am also thankful to J. Stasheff 
for his remarks on the first draft of the paper. 
I would like to thank the anonymous referee for carefully 
reading my manuscript and for many useful suggestions.
A large part of this paper was written when I was 
an assistant professor at the UC Riverside where I 
had excellent working conditions and great 
colleagues. 
I would like to 
thank anonymous taxpayers of California for supporting the 
UC system and my work at UC Riverside.
The results of this paper were presented at the Southern California 
Algebra Conference in UCLA and 
at the Geometry and Physics 
seminar at Northwestern University.  I would like to thank
participants of the conference and the seminar for questions and
useful comments. 
I am partially supported by the NSF grant
DMS 0856196, Regent's Faculty Fellowship, and the Grant for Support
of Scientific Schools NSh-3036.2008.2.

\section{Preliminaries}
\label{sect:prelim}

\subsection{Notation and conventions}
The ground field $\bbK$ has the zero characteristic.
Depending on a context our underlying
symmetric monoidal category is either
the category of graded vector spaces, or
the category of cochain complexes, or the category
of compactly generated topological spaces.
By suspension $\bs V$ of a graded vector space
(or a chain complex) $V$
we mean $\ve \otimes V$, where $\ve$ is a
one-dimensional vector space placed in degree
$+1$\,. For a vector $v\in V$ we denote by $|v|$
its degree. 
$S_n$ denotes the symmetric group on $n$ letters and
$\Sh_{p,q}$ denotes the subset of $(p,q)$-shuffles in
$S_{p+q}$.

For a graded vector space (or a cochain complex) $V$
the notation $S(V)$ (resp. $\tS(V)$) is reserved for the
symmetric algebra (resp. truncated symmetric algebra)
of $V$: 
$$
S(V) = \bbK \oplus V \oplus S^2(V) \oplus S^3(V) \oplus \dots\,, 
$$ 
$$
\tS(V) =  V \oplus S^2(V) \oplus S^3(V) \oplus \dots\,.
$$
Similarly, $T(V)$ (resp. $\tT(V)$) denotes the
tensor algebra (resp. truncated tensor algebra) of $V$:
$$
T(V) = \bbK \oplus V \oplus V^{\otimes  2}  \oplus  V^{\otimes  3} \oplus \dots\,, 
$$
$$
\tT(V) =  V \oplus V^{\otimes  2}  \oplus  V^{\otimes  3}  \oplus \dots\,. 
$$

For an operad $\cO$ (resp. a cooperad $\cC$) and 
a cochain complex $V$ 
we denote by $\cO(V)$ (resp. $\cC(V)$) the free 
$\cO$-algebra generated by $V$ (resp. the cofree 
$\cC$-coalgebra cogenerated by $V$)\,. The notation 
$\Cobar(\cC)$ is reserved for the cobar construction \cite{Fresse}, \cite{GJ} of 
a cooperad $\cC$\,.
 
By ``suspension'' of a (co)operad $\cO$
we mean the (co)operad $\La(\cO)$
whose $m$-th space is
\begin{equation}
\label{susp-op}
\La(\cO)(m) = \bs^{1-m} \cO(m) \otimes \sgn_{m}\,,
\end{equation}
where $\sgn_{m}$ is the sign representation of
the symmetric group $S_m$\,. For example, a $\La\Lie$-algebra
is nothing but a Lie algebra with the bracket of degree $-1$\,.
Similarly $\La\coass$ is the cooperad governing coassociative coalgebras 
with comultiplication of degree $-1$\,.

It is not hard to see that  
for every (co)operad $\cO$ and for every graded
vector space $V$ we have the isomorphism
\begin{equation}
\label{La-O-V}
\La(\cO) (V) = \bs\, \cO(\bs^{-1} V)\,.
\end{equation}

It is important to make a remark about degrees 
of cooperations in coalgebras over cooperads. 
If $\cC$ is a cooperad and $v$ is a homogeneous vector in $\cC(m)$
then the corresponding cooperation
\begin{equation}
\label{map}
A \to A^{\otimes \, m}  
\end{equation}
on a $\cC$-coalgebra $A$ has the degree $-|v|$\,.
In this paper by degree of a cooperation we mean
the degree of the corresponding vector in the cooperad 
but not the degree of the map. For example, a vector 
$\D\in \La\coass(2)$ has the degree $-1$ but the corresponding 
map 
$$
\D : A  \to  A \otimes A
$$    
for a $\La\coass$-coalgebra $A$ has the degree $1$\,.  

In this paper we often encounter $L_{\infty}$-algebras (resp. (DG) Lie algebras) 
with the binary operation having degree $-1$\,. Following our notation we 
call them $\La\Lie_{\infty}$-algebras (resp. $\La\Lie$-algebras). 
However, we still retain the word ``$L_{\infty}$-morphism'' for 
a morphism in the category of  $\La\Lie_{\infty}$-algebras.

Thus a $\La\Lie_{\infty}$-structure on a graded vector space $\cL$ 
is a degree $1$ codifferential $\hQ_{\cL}$ on the cofree $\La^2\cocomm$ 
coalgebra
$$
\La^2\cocomm (\cL)  = \bs^2 S (\bs^{-2} \cL)\,.
$$
Similarly, an $L_{\infty}$-morphism $\hU$ from a $\La\Lie_{\infty}$-algebra
$\cL$ to  a $\La\Lie_{\infty}$-algebra $\tcL$ is a morphism of 
$\La^2\cocomm$-coalgebras
$$
\hU :  \bs^2 S (\bs^{-2} \cL)  \to  \bs^2 S (\bs^{-2} \tcL) 
$$
compatible with the codifferentials. 

Since $\hU$ is a morphism to a cofree coalgebra, it is uniquely 
determined by its composition 
$$
U = p \circ \hU :   \bs^2 S (\bs^{-2} \cL) \to \tcL
$$
with the natural projection 
$$
p:   \bs^2 S (\bs^{-2} \tcL) \to \tcL\,.
$$   

The notation $\Ger$ is reserved for the operad governing Gerstenhaber 
algebras \cite{Ger}, \cite{GJ} and $\Ger^{\vee}$ denotes the corresponding 
Koszul dual \cite{Fresse}, \cite{GJ}, \cite{GK} cooperad. This cooperad 
can be obtained from the linear dual $\Ger^*$ of $\Ger$ by 
applying the operation $\La$ twice: 
$$
\Ger^{\vee} = \La^2 \Ger^*\,.
$$
Following  \cite{Fresse}, \cite{GJ}, and \cite{GK} the DG operad 
$\Cobar(\Ger^{\vee})$ can be chosen as a model for the operad 
which governs homotopy Gerstenhaber algebras. So, in this paper, 
we set 
$$
\Ger_{\infty} = \Cobar(\Ger^{\vee})\,.
$$

For colored (co)operads we use notation and conventions 
from paper \cite{BM-colors} by C. Berger and I. Moerdijk. 
All colored (co)operads in this paper are $2$-colored and 
we denote the two colors by $\mc$ and $\ma$\,. 
Furthermore, we fix the order $\mc < \ma$ on the set $\{\mc, \ma\}$\,.
Hence, due to \cite[Remark 1.3]{BM-colors}, every $2$-colored 
(co)operad $\cO$ is  completely determined 
by objects 
$$
\cO(\underbrace{\mc, \mc, \dots, \mc}_{k},  \underbrace{\ma, \ma, \dots, \ma}_{n}  ;  \mc)\,,
$$
$$
\cO(\underbrace{\mc, \mc, \dots, \mc}_{k},  \underbrace{\ma, \ma, \dots, \ma}_{n}  ;  \ma)
$$ 
together with the right action of the group $S_k \times S_n$
and by the corresponding (co)operadic (co)multiplications 
for these objects. 

We often use the following short-hand notation for 
objects of 2-colored (co)operads
\begin{equation}
\label{notation-mc-ma}
\cO^{\mc}(n,k) : = \cO(\underbrace{\mc, \mc, \dots, \mc}_{n}, 
\underbrace{\ma, \ma, \dots, \ma}_{k}  ;  \mc)\,,
\end{equation}
\begin{equation}
\label{notation-mc-ma1}
\cO^{\ma}(n,k) : = \cO(\underbrace{\mc, \mc, \dots, \mc}_{n}, 
\underbrace{\ma, \ma, \dots, \ma}_{k}  ;  \ma)\,.
\end{equation}
Let us consider, for example, the operad $\Lie$ whose 
algebras are Lie algebras (in the category of vector spaces).
This operad can be upgraded 
to the $2$-colored operad $\Lie^+$ whose algebras are 
pairs $(\cV, \cW)$ where $\cV$ is a Lie algebra and $\cW$ is a module 
over $\cW$\,.
To vectors in $\cV$ (resp, $\cW$) we assign the color $\mc$
(resp. $\ma$). It is not hard to see that for $\Lie^+$ we have 
$$
(\Lie^+)^{\mc}(n,0)= \Lie (n)\,, 
$$
$$
(\Lie^+)^{\ma}(n,1)= \bbK[S_n]\,, 
$$
and the remaining vector spaces are $\bfzero$\,.

\begin{remark}
\label{rem:cofibrant}
For colored operads of cochain 
complexes we use a bit of terminology from closed 
model categories. 
Thus, following \cite{BM} and \cite{BM-colors}, we 
call a map $f: \cO \to  \tcO$ of colored operads
a {\it fibration} if for every collection of colors 
$c_1, \dots, c_k, c$ the map 
$$
f : \cO(c_1, \dots, c_k ; c) \to  \tcO(c_1, \dots, c_k ; c) 
$$
is surjective. We say that  $f: \cO \to  \tcO$
is an {\it acyclic fibration} is $f$ is a quasi-isomorphism and 
a fibration. A map $f: \cO \to  \tcO$ is called 
a {\it cofibration} if $f$ satisfies the left lifting property with 
respect to all acyclic fibrations. Finally, a colored 
operad $\cO$ is called {\it cofibrant} if the unique map 
$\ast \to  \cO$ from the initial object $\ast$ is a cofibration. 

Unfortunately, to the best of our knowledge, it is not yet 
established whether the category of all colored operads 
with the above (co)fibrations is a closed model category. 
For this reason, in our paper, we use the above terminology but 
not the full power of closed model categories.
\end{remark}

We use the reversed grading
on each homological complex. In particular, the homology 
groups $H_{\bul}(X)$ of a topological space $X$ are 
concentrated in non-positive degrees and the 
Poincar\'e duality takes this unusual form:  
\begin{equation}
\label{PD}
H_{\bul}(X) \cong H^{\dim X + \bul}(X)\,.
\end{equation}
Similarly, for the homology spectral sequence associated to 
an increasing filtration $F^{p-1}X \subset F^p X \subset \dots $
on a space $X$ we have
\begin{equation}
\label{E1}
E^1_{p,q}(X) = H_{p+q} (F^{-p} X , F^{-p-1} X)
\end{equation}
and
\begin{equation}
\label{d1}
d_1 :  E^1_{p,q}(X)  \to E^1_{p+1,q}(X)\,.
\end{equation}

\subsection{Hochschild cochain complex of an $A_{\infty}$-algebra}

It is known that for every graded vector space $A$
\begin{equation}
\label{Hom-A-A}
\bigoplus_{k \ge 0} \bs^k\, \Hom(A^{\otimes \, k}, A)
\end{equation}
is equipped with a $\La\Lie$-algebra structure.
The corresponding
bracket is the Gerstenhaber bracket $[\,,\,]_G$\,. It is
given by the formula
$$
[Q_1, Q_2]_{G} =
$$
\begin{equation}
\label{Gerst}
\sum_{i=1}^{k_1}(-1)^{\ve_{i, k_2}}
Q_1(a_1,\,\dots , Q_2 (a_i,\, \dots, a_{i+k_2-1}),\, \dots,
a_{k_1+k_2-1}) -
(-1)^{(k_1 + 1)(k_2 + 1)} (1 \leftrightarrow 2)\,,
\end{equation}
where $k_1$ (resp. $k_2$) is the degree of $Q_1$ (resp. $Q_2$)
in (\ref{Hom-A-A}) and
$$
\ve_{i,k_2} = (k_2+1)(|a_1| + \dots + |a_{i-1}| + i -1 )\,.
$$

Let us recall that an $A_{\infty}$-algebra
structure on $A$ is an element
\begin{equation}
\label{flat}
m \in \bigoplus_{k \ge 1} \bs^k\, \Hom(A^{\otimes \, k}, A)
\end{equation}
satisfying the Maurer-Cartan equation:
\begin{equation}
\label{MC}
[m,m]_{G} = 0\,.
\end{equation}

Equation (\ref{MC}) implies that  the formula
\begin{equation}
\label{Hoch}
\pa^{Hoch} = [m, ~~ ]_{G}
\end{equation}
defines a differential on the $\La\Lie$-algebra (\ref{Hom-A-A})\,. 
In particular, restricting $m$ to $A$ we get a differential on $A$\,. 

The cochain complex (\ref{Hom-A-A}) with the differential 
(\ref{Hoch}) is called the Hochschild
cochain complex of the $A_{\infty}$-algebra
$(A,m)$\,. In this paper we denote this cochain
complex by $\Cbu(A,A)$\,.

\section{The operad $\Ger^+_{\infty}$}
\label{sect:Ger-plus}
Let us start the construction of the operad  $\Ger^+_{\infty}$
with the following definition
\begin{defi}
\label{dfn:S-coalg}
An $\bS$-coalgebra is a pair of $\Ger^{\vee}$-coalgebra
$V$ and a $\La\coass$-coalgebra $A$ together with a unary
operation
$$
\rho : A \to V
$$
of degree\footnote{Recall that by degree of $\rho$ we mean the degree 
of the corresponding vector in the cooperad governing $\bS$-coalgebras.} 
$-1$ such that the composition of
the cobracket $\de_{\mc}$ with $\rho$ is zero:
\begin{equation}
\label{de-rho}
\de_{\mc} \circ  \rho (a) = 0\,,
\end{equation}
and the following diagrams commute:
\begin{equation}
\label{rho-map-coalg}
\xymatrix@M=0.4pc{
A\ar[r]^{\rho} \ar[d]_{\D_{\ma}}  &  V \ar[d]^{\D_{\mc}}  \\
A\otimes A \ar[r]^{\rho \otimes \rho} & V \otimes V
}
\end{equation}
\begin{equation}
\label{left-right-via-rho}
\xymatrix@M=0.4pc{
~ & A \otimes A \ar[r]^{\rho \otimes \id} & V\otimes A\\
A \ar[ur]^{\D_{\ma}} \ar[dr]_{\D_{\ma}} & ~ & ~ \\
~ & A \otimes A \ar[r]^{\id \otimes \rho} & A\otimes V \ar[uu]^{\si}
}
\end{equation}
where $\D_{\mc}$ (resp. $\D_{\ma}$) is the comultiplication in $V$
(resp. in $A$) and
$$
\si: A \otimes V \to V \otimes A
$$
is the mapping which switches
tensor components.
\end{defi}
\begin{remark}
 Diagram (\ref{rho-map-coalg}) says that
$\rho$ is a map of coassociative coalgebras. However,
the degrees of comultiplications $\D_{\mc}$ and $\D_{\ma}$ are different.
This is why $\rho$ is forced to have degree $-1$.
\end{remark}

Composing the comultiplication $\D_{\ma}$ with
$\rho \otimes \id$ and $\id \otimes \rho$ we get the
maps
\begin{equation}
\label{mu-left}
\mu_l = \rho \otimes \id \circ \D_{\ma} : A \to V\otimes A\,,
\end{equation}
and
\begin{equation}
\label{mu-right}
\mu_r = \id \otimes \rho \circ \D_{\ma} : A \to A \otimes V\,,
\end{equation}
respectively.

Using the axioms of $\bS$-coalgebras it is not hard to
show that
\begin{pred}
\label{prop:comodules}
If we view $V$ as a coassociative coalgebra then
the maps  $\mu_l$ (\ref{mu-left}) and $\mu_r$ (\ref{mu-right})
are left and right comodule structures on $A$ over $V$,
respectively.
\end{pred}
\begin{remark}
Axiom (\ref{left-right-via-rho})
 implies that the comodule
structure $\mu_r$ is obtained from $\mu_l$ via
switching the tensor factors.
\end{remark}
\begin{proof}
To prove that $\mu_l$ is a left comodule
structure we need to show that the diagram
\begin{equation}
\label{mu-l-comodule}
\xymatrix@M=0.4pc{
A \ar[d]_{\mu_l} \ar[r]^{\mu_l}  &   V \otimes A \ar[d]^{\D_{\mc} \otimes \id} \\
V \otimes A \ar[r]^-{\id \otimes \mu_l} & V \otimes V \otimes A
}
\end{equation}
commutes.

The latter follows easily from the
commutativity of the diagram below:
\begin{equation}
\label{mu-l-comodule-proof}
\xymatrix@M=0.6pc{
A \ar[d]_{\D_{\ma}} \ar[r]^{\D_{\ma}} &  A\otimes A \ar[r]^{\rho\otimes 1}
\ar[d]^{\D_{\ma} \otimes 1}  &       V \otimes A\ar[d]^{\D_{\mc} \otimes 1} \\
A\otimes A  \ar[r]^-{1\otimes \D_{\ma}} \ar[d]_{\rho \otimes 1}  &
A \otimes A \otimes A \ar[r]^{\rho \otimes \rho \otimes 1}
\ar[d]_{\rho\otimes 1 \otimes 1} &
V \otimes V \otimes A \\
V \otimes A \ar[r]^-{1 \otimes \D_{\ma}} &
V \otimes A  \otimes A \ar[ur]_{1 \otimes \rho \otimes 1} &  ~
}
\end{equation}

The corresponding statement for $\mu_r$ can be proved
in a similar way.

\end{proof}

Let us denote by $\bS$ the cooperad which governs
the $\bS$-coalgebras. This is a 2-colored cooperad.
We reserve the color $\mc$ for vectors in $V$ and the
color\footnote{This notation for colors comes from string theory: $\mc$ stands for {\bf c}losed
strings and $\ma$ stands for {\bf o}pen strings.} 
$\ma$ for vectors in $A$\,. In Subsection \ref{subsect:S-versus-sc}
we describe a link between the cooperad $\bS$ and the cohomology 
cooperad $H^{\bul}(\SC)$ of Voronov's Swiss Cheese operad $\SC$\,.

The cofree $\bS$-coalgebra $\bS(V,A)$ cogenerated by the
pair $(V,A)$ is the direct sum
\begin{equation}
\label{bS-V-A}
\bS(V, A) = \bS^{\mc}(V, A) \,\oplus\,
\bS^{\ma}(V, A)
\end{equation}
where
\begin{equation}
\label{bS-V-A-mc}
\bS^{\mc}(V, A) = \Ger^{\vee}(V)
\end{equation}
and
\begin{equation}
\label{bS-V-A-ma}
\bS^{\ma}(V, A) =
\bs\,  \big( \tS(\bs^{-2} V) \otimes   T(\bs^{-1}A)  \big)
\, \oplus  \, \bs \, \tT(\bs^{-1} A)\,,
\end{equation}
where $\tT(\bs^{-1}A)$  (resp. $\tS(\bs^{-2} V)$) is
the truncated tensor algebra of $\bs^{-1} A$
(resp. the truncated symmetric algebra of $\bs^{-2} V$).

The main hero of this article is the DG operad
$\Cobar(\bS)$\,. Since $\Cobar(\bS)$ extends the
operad governing homotopy Gerstenhaber algebras we denote
it by $\Ger^{+}_{\infty}$\,.

Let us recall that, due to Proposition 2.15 in \cite{GJ},
a $\Ger^{+}_{\infty}$-algebra structure on $(V,A)$ is
a degree $1$ codifferential $\hQ$ of the cofree coalgebra
$\bS(V,A)$\,.

This observation motivates
the definition of a $\Ger_{\infty}^+$-morphism.
\begin{defi}
\label{dfn:Ger-plus-morph}
Let $(V,A)$ and $(V',A')$ be two $\Ger_{\infty}^+$-algebras.
A {\em $\Ger_{\infty}^+$-morphism} from $(V,A)$ to $(V',A')$
is a morphism of $\bS$-coalgebras
$$
\hT: \bS(V,A) \to \bS(V',A')
$$
with the codifferentials $\hQ$ and $\hQ'$ corresponding
to the $\Ger^+_{\infty}$-structures on $(V,A)$
and $(V', A')$, respectively.
\end{defi}

Since $\bS(V',A')$ is cofree every morphism
$$
\hT: \bS(V,A) \to \bS(V',A')
$$
of $\bS$-coalgebras is uniquely determined by
its composition
$$
T = p \circ \hT : \bS(V,A) \to V' \oplus A'
$$
with the projection $p:  \bS(V',A') \to  V' \oplus A'$\,.

It is not hard to see that the compatibility
of $\hT$ with the codifferentials $\hQ$ and $\hQ'$ implies that
the restrictions:
\begin{equation}
\label{T-1-V}
T \Big|_{V} : V \to V'
\end{equation}
and
\begin{equation}
\label{T-1-A}
T \Big|_{A} : A \to A'
\end{equation}
are morphisms of cochain complexes.

This observation motivates the definition of
$\Ger_{\infty}^+$ quasi-isomorphism
\begin{defi}
\label{dfn:Ger-plus-q-iso}
A {\em $\Ger_{\infty}^+$ quasi-morphism} from $(V,A)$ to $(V',A')$
is a morphism of $\bS$-coalgebras
$$
\hT: \bS(V,A) \to \bS(V',A')
$$
for which the restrictions (\ref{T-1-V}), (\ref{T-1-A})
are quasi-isomorphisms of cochain complexes.
\end{defi}

\section{Algebras over the operad $\Ger^{+}_{\infty}$}
\label{sect:Ger-plus-alg}
The following theorem gives us an alternative 
description\footnote{I would like to thank Thomas Willwacher who 
suggested to me this alternative description.}
 of algebras over the operad $\Ger^{+}_{\infty}$\,. 
\begin{teo}
\label{thm:Ger-plus-alg}
A $\Ger^{+}_{\infty}$-algebra structure on the pair
of cochain complexes $(V,A)$ is a triple:
\begin{itemize}

\item A $\Ger_{\infty}$-structure on $V$\,,

\item An $A_{\infty}$-structure on $A$\,, and

\item An $L_{\infty}$-morphism from $V$ to the Hochschild
cochain complex $\Cbu(A,A)$ of the $A_{\infty}$-algebra $A$\,.

\end{itemize}

\end{teo}

\begin{proof}
A $\Ger^{+}_{\infty}$-algebra structure on $(V,A)$
is a degree $1$ codifferential $\hQ$
on the cofree coalgebra $\bS(V,A)$\,.
This differential splits into the $\mc$-component $\hQ^{\mc}$
\begin{equation}
\label{hQ-mc}
\hQ^{\mc} : \Ger^{\vee}(V) \to \Ger^{\vee}(V)
\end{equation}
and the $\ma$-component
\begin{equation}
\label{hQ-ma}
\hQ^{\ma} : \bS^{\ma}(V,A) \to \bS^{\ma}(V,A)\,.
\end{equation}

Since $\bS(V,A)$
is cofree, the codifferential $\hQ$ is uniquely determined
by its composition
\begin{equation}
\label{Q}
Q = p \circ \hQ : \bS(V,A) \to V \oplus A
\end{equation}
with the projection
$$
p : \bS(V,A) \to V \oplus A\,.
$$
The composition $Q$ also splits into the
$\mc$-component
\begin{equation}
\label{Q-mc}
Q^{\mc} : \Ger^{\vee}(V) \to V
\end{equation}
and the $\ma$-component
\begin{equation}
\label{Q-ma}
Q^{\ma} : \bs \, \big(\, \tS(\bs^{-2} V) \otimes   T(\bs^{-1}A) \, \big)
\,\oplus \, \bs \, \tT(\bs^{-1} A)
\to A\,.
\end{equation}

We claim that $Q^{\mc}$ (\ref{Q-mc}) gives us
a $\Ger_{\infty}$-algebra structure on $V$,
the restriction
\begin{equation}
\label{Q-ma-restr}
Q^{\ma} \Big|_{ \bs \tT(\bs^{-1}A) } :  \bs \tT(\bs^{-1}A) \to A
\end{equation}
gives us an $A_{\infty}$-structure on $A$ and
the restriction
\begin{equation}
\label{Q-ma-restr1}
Q^{\ma} \Big|_{ \bs \, \big(\, \tS(\bs^{-2} V) \otimes   T(\bs^{-1}A) \, \big) }
: \bs \, \big(\, \tS(\bs^{-2} V) \otimes  T(\bs^{-1}A) \, \big) \to  A
\end{equation}
gives us an $L_{\infty}$-morphism from the $\La\Lie_{\infty}$-algebra
$V$ to the $\La\Lie$-algebra $C^{\bul}(A,A)$\,.

The first two statements are obvious. Indeed, the map $Q^{\mc}$ (\ref{Q-mc})
corresponds to the codifferential $\hQ^{\mc}$ on the $\Ger^{\vee}$-coalgebra
$\Ger^{\vee}(V)$. This is a $\Ger_{\infty}$-structure on $V$\,.
The map (\ref{Q-ma-restr}) corresponds to the codifferential
$$
\hQ^{\ma} \Big|_{ \bs \tT(\bs^{-1}A) } :  \bs \tT(\bs^{-1}A) \to
 \bs \tT(\bs^{-1}A)
$$
on the cofree $\La\coass$-coalgebra
$$
\La\coass(A) =  \bs \tT(\bs^{-1}A)\,.
$$
Thus we get an $A_{\infty}$-algebra structure on $A$\,.

To prove the third statement we replace
the $\ma$-component $\bS^{\ma}(V, A)$ of the coalgebra
$\bS(V,A)$ by its desuspension
\begin{equation}
\label{desusp-bS-ma}
\bs^{-1} \bS^{\ma}(V, A) = \tS(\bs^{-2} V) \otimes  T(\bs^{-1}A) ~ \oplus ~
 \tT(\bs^{-1} A )
\end{equation}
and conjugate the $\ma$-component $\hQ^{\ma}$ of
the codifferential $\hQ$ by the suspension isomorphism.
(Thanks to this technical modification we will have simpler 
expressions for sings in our formulas.) 
We will use the same notation $Q^{\ma}$ for this conjugated map:
\begin{equation}
\label{Q-ma-conjugated}
Q^{\ma} :   \, \big(\, \tS(\bs^{-2} V) \otimes   T(\bs^{-1}A) \, \big)
\, \oplus  \, \tT(\bs^{-1} A)
\to \bs^{-1} \, A\,.
\end{equation}

Due to the above shift the comultiplication $\D_{\ma}$
on (\ref{desusp-bS-ma}) has degree $0$ and the explicit formula
for $\D_{\ma}$ reads
\begin{equation}
\label{D-ma}
\begin{array}{c}
\displaystyle
\D_{\ma} (v_1, \dots, v_k; a_1, \dots, a_n) = \\[0.3cm]
\displaystyle
\sum_{(p,t) \in E_{k,n}}
\sum_{\la\in \Sh_{p,k-p}} (-1)^{\eta^{\la}_{p,t}}
(v_{\la(1)}, \dots, v_{\la(p)}; a_1, \dots, a_t)
\otimes \\[0.3cm]
\displaystyle
(v_{\la(p+1)}, \dots, v_{\la(k)}; a_{t+1}, \dots, a_n)
\end{array}
\end{equation}
where
$$
E_{k,n} = \{0,1, \dots, k\} \times  \{0,1, \dots, n\}
\setminus \{(0,0), (k,n)\}\,,
$$
\begin{equation}
\label{eta-la-p-t}
\eta^{\la}_{p,t} =  \sum_{i<j ~ \la(i) > \la(j)}
|v_i||v_j |  ~ +  ~ \sum_{i=1}^t \sum_{j=p+1}^k
|v_{\la(j)}| (|a_i| + 1)\,.
\end{equation}

Using the compatibility of $\hQ$ with the
$\bS$-coalgebra structure one deduces readily a formula
for $\hQ^{\ma}$ in terms of $Q^{\mc}$ and $Q^{\ma}$:
\begin{equation}
\label{Q-hQ}
\begin{array}{c}
\displaystyle
\hQ^{\ma}(v_1, \dots, v_k; a_1, \dots, a_n) = \\[0.3cm]
\displaystyle
\sum_{p =1}^{k}\sum_{\la  \in \Sh_{p,k-p}}
(-1)^{\ve_{\la}}(Q^{\mc} (v_{\la(1)}, \dots, v_{\la(p)}), v_{\la(p+1)},
\dots, v_{\la(k)} ; a_1, \dots, a_n ) +
\end{array}
\end{equation}
$$
\sum_{1\le t \le s \le n}^{0 \le p \le k}
\sum_{\la \in \Sh_{p,k-p}}
(-1)^{\ve_{\la} + \ve^{\la}_{t,p}}(v_{\la(1)}, \dots, v_{\la(p)} ;
a_1, \dots, Q^{\ma}(v_{\la(p+1)},
\dots, v_{\la(k)}; a_{t}, \dots, a_{s} ) \dots, a_n ) +
$$
$$
\sum_{1 \le t \le n+1}^{0 \le p \le k-1}
\sum_{\la \in \Sh_{p,k-p}}
(-1)^{\ve_{\la} + \ve^\la_{t,p}}(v_{\la(1)}, \dots, v_{\la(p)}, ;
a_1, \dots, a_{t-1}, Q^{\ma}(v_{\la(p+1)},
\dots, v_{\la(k)}), a_{t}, \dots, a_n )
$$
where 
\begin{equation}
\label{ve-la}
\ve_{\la} = \sum_{i<j ~ \la(i) > \la(j)}
|v_i||v_j |\,,
\end{equation}
and
\begin{equation}
\label{ve-la-tp}
\ve^{\la}_{t,p} = \sum_{j=1}^{p} |v_{\la(j)}| +
\sum_{j=p+1}^{k} \sum_{l=1}^{t-1} |v_{\la(j)}| (|a_l|+1) +
\sum_{l=1}^{t-1} (|a_l|+1)\,.
\end{equation}

Restricting the codifferential $\hQ^{\mc}$ to the
coalgebra
\begin{equation}
\label{cocomm-V}
\La^2\cocomm(V) \subset \Ger^{\vee}(V)
\end{equation}
we get an $\La\Lie_{\infty}$-algebra structure on $V$\,.

An $L_{\infty}$-morphism from $V$ to $\Cbu(A,A)$ is a degree $0$
map 
\begin{equation}
\label{U}
U : \La^2 \cocomm (V) \to \Cbu(A,A)
\end{equation}
satisfying the coherence equations
\begin{equation}
\label{UandQ}
\begin{array}{c}
\displaystyle
\sum_{p=1}^k \sum_{\la  \in \Sh_{p,k-p}}
(-1)^{\ve_{\la}} U(Q^{\mc} (v_{\la(1)}, \dots, v_{\la(p)}), v_{\la(p+1)},
\dots, v_{\la(k)})\\[0.3cm]
\displaystyle
= -[m , U (v_1, \dots, v_k)]_G \, - \\[0.3cm]
\displaystyle
\sum_{p=1}^{k-1} \sum_{\la  \in \Sh_{p,k-p}}
(-1)^{\ve_{\la} + |v_{\la(1)}| + \dots + |v_{\la(p)}|}
[U(v_{\la(1)}, \dots, v_{\la(p)}) ,
U(v_{\la(p+1)}, \dots, v_{\la(k)}) ]_G\,,
\end{array}
\end{equation}
where $\ve_{\la}$ is defined in (\ref{ve-la}), $[\,,\,]_G$ 
is the Gerstenhaber bracket (\ref{Gerst}),
and 
 $m$ is the $A_{\infty}$-structure on $A$
obtained from $Q^{\ma}$, that is
\begin{equation}
\label{m}
m(a_1, a_2, \dots, a_n) = Q^{\ma} (a_1, a_2, \dots, a_n)\,.
\end{equation}

Now we define the desired  $L_{\infty}$-morphism from $V$ to $\Cbu(A,A)$ 
by the formula:   
\begin{equation}
\label{define-U}
U(v_1, \dots, v_k)(a_1, \dots, a_n) =
Q^{\ma} (v_1, \dots, v_k; a_1, \dots a_n)\,,
\end{equation}
where $v_i \in V$ and $a_j \in A$\,.

We claim that equation (\ref{UandQ}) follows from 
\begin{equation}
\label{Q-nilpotent}
\hQ^{\ma} \circ \hQ^{\ma}  = 0\,.
\end{equation}

Indeed, applying the composition $Q^{\ma} \circ  \hQ^{\ma}$ to 
 the monomial
$$
(v_1, v_2, \dots, v_k, a_1, a_2, \dots, a_n)
$$
with $k\ge 1$ and using (\ref{Q-hQ}) we get
\begin{equation}
\label{Q-nilpotent-apply}
\begin{array}{c}
\displaystyle
\sum_{p = 1}^k \sum_{\la  \in \Sh_{p,k-p}}
(-1)^{\ve_{\la}} Q^{\ma} (Q^{\mc} (v_{\la(1)}, \dots, v_{\la(p)}), v_{\la(p+1)},
\dots, v_{\la(k)} ; a_1, \dots, a_n ) +
\end{array}
\end{equation}
$$
\sum_{1\le t \le s \le n}^{0 \le p \le k}
\sum_{\la \in \Sh_{p,k-p}}
(-1)^{\ve_{\la} + \ve^\la_{t,p}}
Q^{\ma}
(v_{\la(1)}, \dots, v_{\la(p)}, ;
a_1, \dots, Q^{\ma}(v_{\la(p+1)},
\dots, v_{\la(k)}; a_t, \dots, a_s ), \dots, a_n )+
$$
$$
\sum_{1\le t \le n+1}^{0 \le p \le k-1}
\sum_{\la \in \Sh_{p,k-p}}
(-1)^{\ve_{\la} + \ve^{\la}_{t,p}}
Q^{\ma}
(v_{\la(1)}, \dots, v_{\la(p)}, ;
a_1, \dots, a_{t-1}, Q^{\ma}(v_{\la(p+1)},
\dots, v_{\la(k)}), a_{t}, \dots, a_n )\,,
$$
where $\ve_{\la}$ and $\ve^{\la}_{t,p}$ are
defined in (\ref{ve-la}) and (\ref{ve-la-tp}),
respectively.

It is not hard to see that
the first sum in (\ref{Q-nilpotent-apply}) coincides
with the left hand side of (\ref{UandQ})\,.
Furthermore, combining all the terms with $p=0$ or $p=k$ in the
second sum of (\ref{Q-nilpotent-apply}) and all the terms
with $p=0$ in the third sum of (\ref{Q-nilpotent-apply})
we get
$$
[m, U(v_1, \dots, v_k)]_G (a_1, \dots, a_n)\,.
$$
Combining the remaining terms in the second sum and in the
third sum of (\ref{Q-nilpotent-apply}) we get
$$
\sum_{p=1}^{k-1}\sum_{\la  \in \Sh_{p,k-p}}
(-1)^{\ve_{\la} + |v_{\la(1)}| + \dots + |v_{\la(p)}|}
[U(v_{\la(1)}, \dots, v_{\la(p)}) ,
U(v_{\la(p+1)}, \dots, v_{\la(k)}) ]_G\,.
$$

Thus we proved that a codifferential $\hQ$ on the coalgebra
$\bS(V,A)$ gives us a $\Ger_{\infty}$-algebra structure on
$V$, an $A_{\infty}$-structure on $A$, and an $L_{\infty}$-morphism
from $V$ to the $\La\Lie$-algebra $\Cbu(A,A)$\,.

Furthermore, given a $\Ger_{\infty}$-algebra structure
$\cD$ on $V$, an $A_{\infty}$-structure $m$ on $A$, and
an $L_{\infty}$-morphism $U$ from $V$ to $\Cbu(A,A)$
the formulas
$$
Q^{\mc} = \cD\,,
$$
(\ref{m}), (\ref{define-U}) define a degree $1$
codifferential $\hQ$ on $\bS(V,A)$\,.

Theorem \ref{thm:Ger-plus-alg} is proved.

\end{proof}

\subsection{$\Ger^+_{\infty}$ versus the operad $\OC_{\infty}$ governing open-closed 
homotopy algebras (OCHA)} 

Inspired by Zwiebach's open-closed string field theory \cite{Zwiebach} 
H. Kajiura and J. Stasheff introduced in \cite{OCHAStasheff}
open-closed homotopy algebras (OCHA). These homotopy 
algebras were studied in papers \cite{OCHA}, \cite{OCHA1}, 
 \cite{OCHAStasheff1}, and \cite{OCHAStasheff11}. 
 In this subsection we show that the operad $\Ger^+_{\infty}$
 is a natural extension of the operad $\OC_{\infty}$\,. 

To recall the definition of OCHA we introduce $\OC^{\vee}$-coalgebras.

\begin{defi}
\label{dfn:OC-vee}
An $\OC^{\vee}$-coalgebra is a pair of a $\La^2\cocomm$-coalgebra
$V$ and a $\La\coass$-coalgebra $A$ together with a unary
operation
$$
\rho : A \to V
$$
of degree $-1$ such that the following diagrams commute:
\begin{equation}
\label{rho-map-coalg-OC}
\xymatrix@M=0.4pc{
A\ar[r]^{\rho} \ar[d]_{\D_{\ma}}  &  V \ar[d]^{\D_{\mc}}  \\
A\otimes A \ar[r]^{\rho \otimes \rho} & V \otimes V
}
\end{equation}
\begin{equation}
\label{left-right-via-rho-OC}
\xymatrix@M=0.4pc{
~ & A \otimes A \ar[r]^{\rho \otimes \id} & V\otimes A\\
A \ar[ur]^{\D_{\ma}} \ar[dr]_{\D_{\ma}} & ~ & ~ \\
~ & A \otimes A \ar[r]^{\id \otimes \rho} & A\otimes V \ar[uu]^{\si}
}
\end{equation}
where $\D_{\mc}$ (resp. $\D_{\ma}$) is the comultiplication in $V$
(resp. in $A$) and
$$
\si: A \otimes V \to V \otimes A
$$
is the mapping which switches
tensor components.
\end{defi}
We denote by $\OC^{\vee}$ the cooperad  which governs 
$\OC^{\vee}$-coalgebras. 

The operad $\OC_{\infty}$ governing OCHA algebras is by definition 
the cobar construction of the cooperad $\OC^{\vee}$. 
\begin{equation}
\label{OC-infinity}
\OC_{\infty} = \Cobar (\OC^{\vee})\,.
\end{equation}  
In other words, an OCHA algebra structure on a pair 
of graded vector spaces $(V,A)$ is a degree $1$ codifferential   
on the free $\OC^{\vee}$- coalgebra $\OC^{\vee}(V,A)$\,. 

Comparing Definitions  \ref{dfn:S-coalg} and \ref{dfn:OC-vee} it is 
not hard to see that $\OC^{\vee}$ is a sub-cooperad of $\bS$\,. 
Hence $\OC_{\infty}$ is a sub DG operad of $\Ger^+_{\infty}$: 
\begin{equation}
\label{OC-inf-in-Ger-plus}
\OC_{\infty} \hookrightarrow \Ger^+_{\infty}\,.
\end{equation} 
Furthermore, Theorem \ref{thm:Ger-plus-alg} is a generalization 
of the following statement from \cite{OCHAStasheff1}
\begin{teo}[H. Kajiura, J. Stasheff, \cite{OCHAStasheff1}]
\label{thm:OCHA}
An $\OC_{\infty}$-algebra structure on the pair
of cochain complexes $(V,A)$ is a triple:
\begin{itemize}

\item A $\La\Lie_{\infty}$-structure on $V$\,,

\item An $A_{\infty}$-structure on $A$\,, and

\item An $L_{\infty}$-morphism from $V$ to the Hochschild
cochain complex $\Cbu(A,A)$ of the $A_{\infty}$-algebra $A$\,.

\end{itemize}

\end{teo}

\begin{example}
\label{OCHA-Cbu-A}
For every $A_{\infty}$-algebra $A$ the pair $(\Cbu(A,A), A)$ carries 
a tautological OCHA structure: the  $\La\Lie_{\infty}$-structure on 
$\Cbu(A,A)$ is simply the $\La\Lie$-algebra structure given by 
the Hochschild differential (\ref{Hoch}) and the Gerstenhaber bracket (\ref{Gerst});
the $A_{\infty}$-structure on $A$ is the given one;  finally, 
the $L_{\infty}$-morphism from $\Cbu(A,A)$ to $\Cbu(A,A)$ is 
the identity map.  
\end{example}

\subsection{Homotopies of $\Ger^+_{\infty}$-algebra structures}
Let $(V,A)$ be a pair of cochain complexes. Let
$\hQ$ and $\hQ'$ be $\Ger^+_{\infty}$-algebra
structures on $(V,A)$\,.
\begin{defi}
\label{dfn:homotopies}
We say that the $\Ger^+_{\infty}$-algebra structures
$\hQ$ and $\hQ'$ are homotopic if there exists a degree
$0$ coderivation $\hpsi$ of the coalgebra $\bS(V,A)$ satisfying
\begin{equation}
\label{conjugated}
\hQ' = \exp([\hpsi, \,]) \hQ\,.
\end{equation}
\end{defi}
\begin{remark}
\label{rem:homot-q-iso}
Homotopic $\Ger^+_{\infty}$-structures on a pair $(V,A)$
are obviously quasi-isomorphic. Namely, a  $\Ger^+_{\infty}$ quasi-isomorphism
from $(V,A, \hQ)$ to $(V,A, \hQ')$ is given by the $\bS$-coalgebra
automorphism $\exp(\hpsi)$ of $\bS(V,A)$\,.
\end{remark}
\begin{remark}
Definition \ref{dfn:homotopies} is justified by Proposition 4.10
from paper\footnote{I would like to thank Bruno Vallette for showing me this paper.} 
\cite{ChuangLazarev} by J. Chuang and A. Lazarev.
Indeed,
let $\End_{(V,A)}$ be the endomorphism operad of the pair $(V,A)$\,.
Then Proposition 4.10 from \cite{ChuangLazarev}
implies  that $\Ger^+_{\infty}$-algebra structures
$\hQ$ and $\hQ'$ are in the relation (\ref{conjugated}) if
and only if the corresponding maps of DG operads
$$
\Cobar(\bS) \to  \End_{(V,A)}
$$
are homotopic with the homotopy defined via a path object. 
\end{remark}

We will be interested in homotopies between $\Ger^+_{\infty}$-algebra
structures for which the coderivation
$$
\hpsi \in \Coder(\bS(V,A))
$$
satisfies the following additional conditions
\begin{equation}
\label{condition-V}
\hpsi \Big|_{\Ger^{\vee}(V)} = 0\,.
\end{equation}
\begin{equation}
\label{condition-A}
\hpsi \Big|_{ \tT(\bs^{-1}A)} = 0\,.
\end{equation}
These conditions are motivated by
the following obvious proposition:
\begin{pred}
\label{prop:conditions}
Let $(V,A)$ be a pair of cochain complexes and
$\hQ$, $\hQ'$ be $\Ger^+_{\infty}$-algebra
structures on $(V,A)$ satisfying equation
(\ref{conjugated}). If the coderivation
$\hpsi$ enjoys condition (\ref{condition-V}) then
the $\Ger_{\infty}$-algebra structures on $V$ corresponding
to $\hQ$ and $\hQ'$ coincide.
If $\hpsi$ enjoys condition (\ref{condition-A}) then
the $A_{\infty}$-algebra structures on $A$ corresponding
to $\hQ$ and $\hQ'$ coincide.
\end{pred}
\begin{proof}
The first statement follows from the fact that the $\Ger_{\infty}$-algebra
structure on $V$ corresponding to $\hQ$
is given by the $\mc$-component (\ref{hQ-mc}) of $\hQ$\,.
The second statement follows from the fact that the $A_{\infty}$-algebra
structure on $A$ corresponding to $\hQ$
is obtained by restricting $Q^{\ma}$ (\ref{Q-ma})
to
$$
\tT(\bs^{-1} A) \subset \bs^{-1} \bS^{\ma}(V,A)\,.
$$
\end{proof}

Let $\cL$ and $\tcL$ be $\La\Lie_{\infty}$-algebras.

Recall from
\cite{Erratum} and \cite{Shoikhet}
that $L_{\infty}$-morphisms from
$\cL$ to $\tcL$ can be identified with Maurer-Cartan elements
of the auxiliary $L_{\infty}$-algebra:
\begin{equation}
\label{aux-L-infty}
\bs \Hom(\bs^2 \tS(\bs^{-2}\cL), \tcL)\,.
\end{equation}
The $L_{\infty}$-structure on (\ref{aux-L-infty})
is defined in terms of the $\La\Lie_{\infty}$-structure on
$\tcL$, the comultiplication and the codifferential
on $\bs^2 \tS(\bs^{-2}\cL)$\,. In particular, if $\tcL$ is
a $\La\Lie$-algebra then (\ref{aux-L-infty}) is a DG Lie algebra.
This is an important observation because in our consideration
below
$\tcL$ is the Hochschild cochain complex $\Cbu(A,A)$ and
hence a $\La\Lie$-algebra.

It was also shown in \cite{Erratum} and \cite{Shoikhet}
that homotopic $L_{\infty}$-morphisms correspond to
isomorphic\footnote{We think of Maurer-Cartan elements as
objects of the corresponding
Deligne groupoid \cite{Ezra-higher}, \cite{Ezra-Lie}.}
Maurer-Cartan elements in (\ref{aux-L-infty}).

We will need the following theorem.
\begin{teo}
\label{thm:homotopies}
Let $(V,A)$ be a pair of cochain complexes and
$\hQ$, $\hQ'$ be $\Ger^+_{\infty}$-algebra
structures on $(V,A)$ with coinciding $\Ger_{\infty}$-structures
on $V$ and coinciding $A_{\infty}$-structures on $A$\,.
Let $\hU$ and $\hU'$ be $L_{\infty}$-morphisms
from $V$ to $\Cbu(A,A)$ corresponding to
$\Ger^+_{\infty}$-algebra
structures $\hQ$ and $\hQ'$, respectively.
Then the $L_{\infty}$-morphisms $\hU$ and $\hU'$ are
homotopic if and only if $\hQ$ and $\hQ'$ satisfy
equation (\ref{conjugated}) with a coderivation $\hpsi$
enjoying conditions (\ref{condition-V})
and (\ref{condition-A}).
\end{teo}
\begin{proof}
We start with the part ``if''.

Let $U$ (resp. $U'$) be the composition of $\hU$
(resp. $\hU'$) with the natural projection
$p:  \bs^2 \tS(\bs^{-2}  \Cbu(A,A) ) \to  \Cbu(A,A)$
\begin{equation}
\label{UU-prime}
U, U' : \bs^2 \tS(\bs^{-2} V) \to  \Cbu(A,A)\,.
\end{equation}
Following \cite{Erratum}, \cite{Shoikhet} we view the
maps $U$ and $U'$ as degree $1$ elements of the
following auxiliary DG Lie algebra
\begin{equation}
\label{aux-DGLA}
\cH = \bs \Hom(\bs^2 \tS(\bs^{-2}V), \Cbu(A,A))\,.
\end{equation}
We need to show that there exists a degree zero
element
\begin{equation}
\label{te-aux-DGLA}
\te \in \bs \Hom(\bs^2 \tS(\bs^{-2}V), \Cbu(A,A))
\end{equation}
such that
\begin{equation}
\label{UU-prime-te}
U'= \exp([\te, \,]_{\cH}) U  +
\frac{ \exp([\te, \,]_{\cH}) - 1}{[\te, \,]_{\cH}} d_{\cH} \te \,,
\end{equation}
where $d_{\cH}$ and $[\,,\,]_{\cH}$ are the differential and the
bracket in $\cH$ (\ref{aux-DGLA}), respectively.

We will show that the desired element $\te$ can
be expressed in terms of the map
$$
\psi = p \circ \hpsi\,,
$$
where $p$ is the natural projection $p : \bS(V,A) \to V \oplus A$\,.
More precisely, for $v_i \in V$ and $a_j \in A$ we set
\begin{equation}
\label{te-define}
\te (v_1, \dots, v_k) (a_1, \dots, a_n) =
\psi (v_1, \dots, v_k ; a_1, \dots, a_n)\,.
\end{equation}
It is easy to check that $\te$ has the degree $0$
in (\ref{aux-DGLA}).

To prove (\ref{UU-prime-te}) we introduce an auxiliary
variable $t$ and the following path of $L_{\infty}$-morphisms
\begin{equation}
\label{U-t}
U_t= \exp(t[\te, \,]_{\cH}) U  +
\frac{ \exp(t[\te, \,]_{\cH}) - 1}{[\te, \,]_{\cH}} d_{\cH} \te \,.
\end{equation}

Similarly, we introduce a path $\hQ_t$ of $\Ger^+_{\infty}$-algebra
structures on the pair $(V,A)$
\begin{equation}
\label{hQ-t}
\hQ_t = \exp(t[\hpsi, \,]) \hQ\,.
\end{equation}
Let $\tU_t$ be the $L_{\infty}$-morphism from
$V$ to $\Cbu(A,A)$ corresponding to $\hQ_t$\,.
Namely,
\begin{equation}
\label{tU-t}
\tU_t(v_1, \dots, v_k) (a_1, \dots, a_n) =
Q_t(v_1, \dots, v_k ; a_1, \dots, a_n)\,,
\end{equation}
where $Q_t$ is the composition $p\circ \hQ_t$
of $\hQ_t$ with the natural projection
$p: \bS(V,A)\to V \oplus A$\,.

The codifferential $\hQ_t$ satisfies the differential
equation
\begin{equation}
\label{hQ-t-diff-eq}
\frac{d}{d t} \hQ_t = [\hpsi, \hQ_t\,]\,.
\end{equation}
Hence the composition $Q_t= p \circ \hQ_t$ satisfies
the equation
\begin{equation}
\label{Q-t-diff-eq}
\frac{d}{d t} Q_t = \psi \circ \hQ_t -
Q_t \circ \hpsi
\end{equation}
Combining this equation with (\ref{te-define})
and (\ref{tU-t}) it is not hard to show that
$\tU_t$ satisfies the differential equation
\begin{equation}
\label{tU-t-diff-eq}
\frac{d}{d t} \tU_t = d_{\cH} \te + [\te, \tU_t]_{\cH}
\end{equation}
where $d_{\cH}$ and $[\,,\,]_{\cH}$ are the differential and the
bracket in $\cH$ (\ref{aux-DGLA}), respectively.

On the other hand, $U_t$ (\ref{U-t}) satisfies the
same differential equation (\ref{tU-t-diff-eq}) with the
same initial condition
$$
U_t \Big|_{t=0} = \tU_t \Big|_{t=0} = U \,.
$$
Thus equation (\ref{UU-prime-te}) holds and the
desired implication is proved. 

The proof of ``only if'' part is similar.

We define the map
$$
\psi: \bS(V,A) \to V \oplus A
$$
by setting
\begin{equation}
\label{psi-V}
\psi \Big|_{\Ger^{\vee}(V)} = 0
\end{equation}
and
\begin{equation}
\label{psi-V-A}
\psi(v_1, v_2, \dots, v_k, a_1, a_2, \dots a_n)=
\begin{cases}
\te(v_1, v_2, \dots, v_k)(a_1, a_2, \dots a_n)\,, {\rm if} ~ k \ge 1\,, \cr
0\,, ~ {\rm otherwise}\,.
\end{cases}
\end{equation}
Next, we let $\hpsi$ be the coderivation corresponding
to $\psi$\,.

Then using a similar line of arguments it is
not hard to deduce equation (\ref{conjugated})
from (\ref{UU-prime-te}).

\end{proof}

\subsection{Transfer theorem}
Let us now show that $\Ger_{\infty}^+$-algebra
structures can be transferred across quasi-isomorphisms.
More precisely,
\begin{teo}
\label{thm:transfer}
Given quasi-isomorphisms of cochain complexes
\begin{equation}
\label{q-isos}
V \to V' \qquad A \to A'
\end{equation}
and a $\Ger^+_{\infty}$-algebra structure on the pair $(V',A')$
one can construct a $\Ger^+_{\infty}$-algebra structure on
$(V,A)$ and a $\Ger^+_{\infty}$ quasi-isomorphism
$$
\hT : (V,A) \brarrow (V',A')
$$
which extends (\ref{q-isos}).
\end{teo}
\begin{remark}
The $\Ger^+_{\infty}$-algebra structure obtained in
this way on pair $(V,A)$ is called the {\it transferred}
$\Ger^+_{\infty}$-algebra structure.
\end{remark}
\begin{remark}
\label{OCHA-transfer}
We would like to remark that OCHA structures can 
also be transferred across quasi-isomorphisms. 
The corresponding transfer theorem was proved 
in \cite[Section 4]{OCHAStasheff}. 
\end{remark}
\begin{proof}
The construction of the transferred $\Ger^+_{\infty}$-algebra
structure goes along the lines of the proof of Lemma 4.2.1
in \cite{Hinich-pro-Dimu}.

A $\Ger^+_{\infty}$-algebra structure on $(V',A')$ is given
by a degree $1$ codifferential $\hQ'$ on the $\bS$-coalgebra
$$
\bS(V',A')\,.
$$
Thus we need to build a coderivation $\hQ$ on the
$\bS$-coalgebra $\bS(V,A)$ together with an $\bS$-coalgebra
homomorphism
$$
\hT : \bS(V,A) \to \bS(V',A')
$$
such that $\hT$ extends (\ref{q-isos}) $\hQ$
extends the differential on $V\oplus A$
and
\begin{equation}
\label{nado-reshat}
\hQ^2 = 0\,, \qquad \hT \hQ = \hQ' \hT \,.
\end{equation}
Following V. Hinich \cite{Hinich-pro-Dimu} these equations
can be solved by induction on an increasing filtration
on the cooperad $\bS$\,. However, unlike in the proof
of Lemma 4.2.1 from \cite{Hinich-pro-Dimu} we cannot use
the filtration by arity because $\bS$ has a non-identity
unary operation $\rho$.

So we assign the following weights to the
elementary cooperations of $\bS$
\begin{equation}
\label{weights}
w(\rho)= 1\,, \qquad w(\D_{\ma}) =1\,, \qquad
w(\D_{\mc}) =2\,, \qquad  w(\de_{\mc}) = 2\,,
\end{equation}
where $\D_{\mc}$ (resp. $\de_{\mc}$) is the comultiplication
(resp. cobracket) of $\Ger^{\vee}$ and $\D_{\ma}$ is the
comultiplication of $\La\coass$\,.

Since all the defining relations of
the cooperad $\bS$ are homogenous with respect
to the weights (\ref{weights}) we use them to build
the desired increasing filtration $F^{\bul}\bS$ on $\bS$
\begin{equation}
\label{bS-filtr}
F^1 \bS \subset F^2 \bS \subset F^3 \bS \subset \dots\,.
\end{equation}
The cooperad $\bS$ is obviously cocomplete with respect
to this filtration. Furthermore, the filtration is compatible
with the coinsertions $\D_k$ of the cooperad $\bS$ in the following
sense\footnote{This is also known as the minimal model condition
\cite{Sullivan}.}
\begin{equation}
\label{minimal}
\D_k (\,F^n \bS\,) \subset \bigoplus_{p+q=n}
F^p \bS \otimes F^q \bS\,.
\end{equation}
It is this property that allows us to solve
equations (\ref{nado-reshat}) inductively following
the lines of the proof of Lemma 4.2.1 from \cite{Hinich-pro-Dimu}.
\end{proof}
\begin{remark}
\label{rem:cofibrant1}
The minimal model condition (\ref{minimal}) implies that
the operad $\Ger^+_{\infty}$ is cofibrant in the sense of 
Remark \ref{rem:cofibrant}.
\end{remark}
\begin{remark}
\label{rem:Laan}
The transfer theorem is a common feature of homotopy 
algebras. As far as we know, the first example of such a 
theorem was proved by T. Kadeishvili in \cite{Kadeishvili} for 
$A_{\infty}$-algebras. In this context, we would also like to mention 
a very interesting paper \cite{Laan} by P. van der Laan. In this paper
the author established koszulity for the
colored operad which governs operads. Using this 
result he introduced the notion of operads up to
homotopy and gave an elegant proof of the
transfer theorem for homotopy algebras.
\end{remark}

\subsection{$\Ger^+_{\infty}$-algebra structure on $(\Cbu(A,A), A)$}
\label{subsec:ona}

Let $A$ be an $A_{\infty}$-algebra.
Following D. Tamarkin \cite{Hinich-pro-Dimu}, \cite{Dima-Proof}
 the Hochschild
cochain complex $\Cbu(A,A)$ of $A$ carries a $\Ger_{\infty}$-algebra
structure with the following properties
\begin{itemize}

\item[{\bf H}] the induced Gerstenhaber algebra structure on
the cohomology groups $H^{\bul}(\Cbu(A,A))$ coincides with the one
given by the Gerstenhaber bracket and the cup-product.

\item[{\bf O}] all operations
of this $\Ger_{\infty}$-algebra structure are expressed in
terms of insertions of cochains into a cochain, insertions of
cochains into $A_{\infty}$ operations, and insertions of 
$A_{\infty}$-operations into a cochain.

\end{itemize}

\begin{remark}
Tamarkin's $\Ger_{\infty}$-algebra structure
on $\Cbu(A,A)$ can be obtained
in two ways. The first way is to combine a solution
of the Deligne conjecture \cite{BF}, \cite{KS}, \cite{M-Smith},
\cite{Dima-DG}, \cite{Sasha}
for Hochschild cochains with
the formality for the operad of little discs \cite{K1},
\cite{Volic}, \cite{Dima-Disk}. The second way is to
use the Etingof-Kazhdan theory \cite{EK} as in \cite{TT}.
Both constructions involve various choices.
For example, for both constructions we need to choose
Drinfeld's associator $\Phi$ \cite{Drinfeld}. Following D. Tamarkin \cite{Dima-GT},
different choices of the Drinfeld associator
give different homotopy classes of the $\Ger_{\infty}$-structure
on $\Cbu(A,A)$\,.
\end{remark}

Let us prove that
\begin{teo}
\label{Ger-plus-on-C-A}
Tamarkin's $\Ger_{\infty}$-algebra structure on
$\Cbu(A,A)$ extends to a $\Ger^+_{\infty}$-algebra structure
on the pair $(\Cbu(A,A), A)$\,.
\end{teo}
\begin{proof}
According to Theorem \ref{thm:Ger-plus-alg} we need to construct
a $\Ger_{\infty}$-algebra structure on $\Cbu(A,A)$, an $A_{\infty}$-algebra
structure on $A$ and an $L_{\infty}$-morphism from
$\Cbu(A,A)$ to the $\La\Lie$-algebra  $\Cbu(A,A)$\,.

The first two ingredients are already at hand.
So we set the last one to be the identity
map
\begin{equation}
\label{id}
\id : \Cbu(A,A) \to \Cbu(A,A)\,.
\end{equation}
We would like to emphasize that the $\La\Lie_{\infty}$-algebra
structure on the source of (\ref{id}) comes from
Tamarkin's $\Ger_{\infty}$-algebra structure while the
target of (\ref{id}) carries the $\La\Lie$-algebra
structure given by the Hochschild differential (\ref{Hoch})
and the Gerstenhaber bracket (\ref{Gerst}).
Due to Theorem\footnote{In paper \cite{DTT}, Theorem 2 was 
proved for the case of (strict) associative algebra $A$ only. 
However it is not hard to extend the proof from \cite{DTT} to the
case of an arbitrary $A_{\infty}$-algebra using the same 
degree bookkeeping argument. }
2 from \cite{DTT} these two $\La\Lie_{\infty}$-algebra
structures on $\Cbu(A,A)$ coincide. Hence,
the map (\ref{id}) is indeed an
 $L_{\infty}$-morphism from the $\La\Lie_{\infty}$-algebra $\Cbu(A,A)$
to the $\La\Lie$-algebra of Hochschild cochains of $A$\,.

\end{proof}

\begin{remark}
\label{rem:OCHA-Cbu-A}
It is easy to see that the $\Ger^+_{\infty}$ structure on the pair $(\Cbu(A,A), A)$ 
extends the tautological OCHA structure from Example \ref{OCHA-Cbu-A}.
\end{remark}

Let us describe the $\Ger_{\infty}^+$-algebra structure
from Theorem \ref{Ger-plus-on-C-A}
explicitly in the case when $A$ is a plain associative
algebra concentrated in degree $0$.

We need to define a codifferential $\hQ$ on
the $\bS$-coalgebra $\bS(\Cbu(A,A), A)$\,. This codifferential
is uniquely determined by its composition
$$
Q = p\circ \hQ
$$
with the projection
$$
p\,:\,  \bS(\Cbu(A,A), A) \to
\Cbu(A,A) \oplus A\,.
$$

The $\mc$-component $Q^{\mc}$ of $Q$
\begin{equation}
\label{mc-Q}
Q^{\mc}: \Ger^{\vee}(\Cbu(A,A)) \to \Cbu(A,A)
\end{equation}
is Tamarkin's $\Ger_{\infty}$-algebra structure
on $\Cbu(A,A)$ and the $\ma$-component is
defined as follows.

For $a_1, \dots, a_n \in A$ and
$P, P_1, \dots, P_k \in \Cbu(A,A)$ we set
\begin{equation}
\label{ma-Q0}
Q^{\ma} (a_1, a_2, \dots, a_n) =
\begin{cases}
- a_1 a_2 \,, {\rm if} ~ n = 2\,, \cr
0\,, ~ {\rm otherwise}\,,
\end{cases}
\end{equation}
\begin{equation}
\label{ma-Q1}
Q^{\ma} (P_1, P_2 \dots, P_k ;a_1, a_2, \dots, a_n) = 0
\end{equation}
if $k > 1 $ and
\begin{equation}
\label{ma-Q11}
Q^{\ma} (P; a_1, a_2, \dots, a_n) =
\begin{cases}
P(a_1, \dots, a_n) \,, {\rm if} ~ |P| = n \,, \cr
0\,, ~ {\rm otherwise}\,.
\end{cases}
\end{equation}
In particular, for $ b \in C^0(A)$ we have
\begin{equation}
\label{ma-Q11-0}
Q^{\ma} (b) = b\,,
\end{equation}
where, in the right hand side, $b$ is viewed as an element
of $A$\,.

\section{The formality quasi-isomorphism for $\Cbu(A,A)$ via transfer
of the $\Ger^+_{\infty}$-structure}
\label{sect:via-transfer}
Let $A$ be a commutative unital algebra over the field $\bbK$
and let
\begin{equation}
\label{V-bul}
V^{\bul}_A = \wedge^{\bul}_A \Der(A)
\end{equation}
be the graded algebra of polyvector fields on $Spec(A)$\,.
It is known that $V^{\bul}(A)$ is a Gerstenhaber algebra.
The multiplication is simply the exterior product and
the Lie bracket is the Schouten-Nijenhuis bracket $[\,,\,]_{SN}$\,.

Let
\begin{equation}
\label{hQ-V-A}
\hQ_H : \bS(V^{\bul}_A, A) \to  \bS(V^{\bul}_A, A)
\end{equation}
be a $\Ger^+_{\infty}$-structure on the pair
$(V^{\bul}_A, A)$\,. As above we denote by
$Q_H$ the composition $p \circ \hQ_H$ of $\hQ_H$ with
the natural projection:
$$
p : \bS(V^{\bul}_A, A) \to  V^{\bul}_A \oplus A\,.
$$
We also denote by $Q^{\mc}_H$ and $Q^{\ma}_H$ the
$\mc$-component and $\ma$-component, respectively
$$
Q^{\mc}_H : \Ger^{\vee}(V^{\bul}_A) \to V^{\bul}_A\,,
$$
$$
Q^{\ma}_H : \bs\,  \big( \tS(\bs^{-2} V^{\bul}_A) \otimes   T(\bs^{-1}A)  \big)
\, \oplus  \, \bs \, \tT(\bs^{-1} A) \to A\,.
$$

Theorem \ref{thm:Ger-plus-alg} implies that
\begin{cor}
\label{Ger-plus-gives-F}
If a $\Ger^+_{\infty}$ algebra structure (\ref{hQ-V-A})
on the pair $(V^{\bul}_A, A)$ satisfies these two properties:
\begin{itemize}

\item[{\bf G}] the $\Ger_{\infty}$-structure on $V^{\bul}_A$  is the
standard Gerstenhaber algebra structure given by the Schouten-Nijenhuis 
bracket and the exterior product;

\item[{\bf A}] the $A_{\infty}$-algebra structure on $A$ is
the original associative algebra structure on $A$

\end{itemize}
then the formula
\begin{equation}
\label{formula}
F(\ga_1, \dots, \ga_k) (a_1, \dots, a_n) =
Q_H^{\ma}(\ga_1, \dots, \ga_k ; a_1, \dots, a_n)
\end{equation}
defines an $L_\infty$-morphism from the $\La\Lie$-algebra
$V^{\bul}_A$ to the $\La\Lie$-algebra $\Cbu(A,A)$\,.

Conversely, if
\begin{equation}
\label{hF}
\hF : V^{\bul}_A \brarrow \Cbu(A,A)
\end{equation}
is an $L_{\infty}$-morphism
from the $\La\Lie$-algebra $V^{\bul}_A$ to $\Cbu(A,A)$ then
the formula (\ref{formula}) (with $F= p \circ \hF$)
defines a $\Ger^+_{\infty}$-structure on the pair $(V^{\bul}_A, A)$
satisfying the above properties {\bf G} and {\bf A}.
\end{cor}
We will denote the $\Ger^+_{\infty}$-structure on the
pair $(V^{\bul}_A, A)$ corresponding to an $L_{\infty}$-morphism
(\ref{hF}) by $\hQ_{\hF}$\,.

\begin{remark} 
\label{rem:OCHA-and-Kontsevich}
An OCHA analogue of Corollary  \ref{Ger-plus-gives-F} was 
already proved in \cite{OCHAStasheff1}. In fact, J. Kajiura and 
J. Stasheff considered in \cite{OCHAStasheff1} the OCHA structure on the pair    
$( V^{\bul}_A, A)$ for $A = \bbR[x^1, \dots, x^d]$ corresponding 
to Kontsevich's formality quasi-isomorphism from \cite{K}\,. 
\end{remark}

Let us now restrict our consideration  to the case
when $A$ is a regular commutative algebra.
(In other words, for every prime ideal $\mp\subset A$
the local ring $A_{\mp}$ is regular.)
In this case,
the Hochschild-Kostant-Rosenberg theorem \cite{HKR}
tells us that
the Hochschild cohomology $HH^{\bul}(A,A)$ is isomorphic
to the Gerstenhaber algebra (\ref{V-bul}) of polyvector
fields on $\Spec(A)$\,.

For regular commutative algebras we also
have the following theorem.
\begin{teo}[Theorem 4, \cite{DTT}]
\label{thm:nasha}
Let $\cD$ be Tamarkin's $\Ger_{\infty}$-structure on $\Cbu(A,A)$\,.
Then for every regular commutative algebra $A$ (over $\bbK$)
there exists a $\Ger_{\infty}$ quasi-isomorphism
\begin{equation}
\label{hG}
\hG : V^{\bul}_A \brarrow \Cbu(A,A)
\end{equation}
from the Gerstenhaber algebra $V^{\bul}_A$ of polyvector
fields to the $\Ger_{\infty}$-algebra $(\Cbu(A,A), \cD)$\,.
\end{teo}
Due to Theorem 2 from \cite{DTT} the $\La\Lie_{\infty}$-part
of the $\Ger_{\infty}$-structure $\cD$ on $\Cbu(A,A)$ coincides
with the standard $\La\Lie$-algebra structure on $\Cbu(A,A)$\,.
Therefore, restricting the $\Ger_{\infty}$-quasi-isomorphism $\hG$ to
the cocommutative coalgebra
$$
\La^2\cocomm ( V^{\bul}_A) = \bs^2\, \tS(\bs^{-2} V^{\bul}_A )
$$
we get an $L_{\infty}$ quasi-isomorphism
\begin{equation}
\label{L-infty}
\hG \Big|_{\bs^2 S(\bs^{-2} V^{\bul}_A )}
 \bs^2 S(\bs^{-2} V^{\bul}_A ) \to  \bs^2 S(\bs^{-2} \Cbu(A,A))
\end{equation}
from the $\La\Lie$-algebra $V^{\bul}_A$ to the $\La\Lie$-algebra $\Cbu(A,A)$.

The following theorem addresses the question of whether
a given $L_{\infty}$-morphism $\hF$ (\ref{hF}) is homotopic
to the $L_{\infty}$ quasi-isomorphism (\ref{L-infty}).

\begin{teo}
\label{Teorema}
Let $\hQ_{\hF}$ be the $\Ger^+_{\infty}$-structure
on the pair $(V^{\bul}_A, A)$ corresponding
to an $L_{\infty}$-morphism $\hF$ (\ref{hF}).
Furthermore, let $\hQ$ be the $\Ger^+_{\infty}$-structure
on the pair $(\Cbu(A,A), A)$ which extends Tamarkin's
$\Ger_{\infty}$-structure $\cD$ on $\Cbu(A,A)$ via
Theorem \ref{Ger-plus-on-C-A}. Then the $L_{\infty}$-morphism
$\hF$ is homotopic to the $L_{\infty}$ quasi-isomorphism
(\ref{L-infty}) if and only if the $\Ger^+_{\infty}$-algebras
$(V^{\bul}_A, A, \hQ_{\hF})$ and $(\Cbu(A,A), A, \hQ)$
are quasi-isomorphic.
\end{teo}

\begin{proof}
We start with the part ``only if''.

Let $\hF$ be homotopic to
the $L_{\infty}$-quasi-isomorphism (\ref{L-infty}).
Due to
Theorem \ref{thm:homotopies} and Remark \ref{rem:homot-q-iso}
we may assume, without loss of generality, that
$\hF$ coincides with an $L_{\infty}$-quasi-isomorphism (\ref{L-infty}).
In other words,
\begin{equation}
\label{hG-hF}
\hG \Big |_{\bs^{2}\, \tS(\bs^{-2} V^{\bul}_A) } = \hF\,.
\end{equation}

Then we define a morphism
$\hT$ from the $\bS$-coalgebra
$\bS(V^{\bul}_A, A)$ to $\bS(\Cbu(A,A), A)$
by setting
\begin{equation}
\label{hT-V}
\hT \Big|_{\Ger^{\vee}(V^{\bul}_A)} = \hG
\end{equation}
and
\begin{equation}
\label{T-ma}
T^{\ma}(\ga_1, \dots, \ga_k, a_1, \dots, a_n)=
\begin{cases}
a_1 ~~{\rm if}~ ~ k = 0 ~{\rm and}~ n = 1\,, \cr
0\,, ~ {\rm otherwise}\,,
\end{cases}
\end{equation}
where
$$
T^{\ma} : \bs \, \big(\tS (\bs^{-2} V^{\bul}_A) \,\otimes\, T(\bs^{-1} A)  \big)
\,\oplus \, \bs \tT(\bs^{-1}A) \to A
$$
is the $\ma$-component of the composition $T = p \circ \hT$\,.

Equations (\ref{hG-hF}) and (\ref{hT-V}) imply that
\begin{equation}
\label{hF-hT-mc}
F (\ga_1, \dots, \ga_k) = T^{\mc} (\ga_1, \dots, \ga_k)
\end{equation}
for all $\ga_i \in V^{\bul}_A$\,. Here $T^{\mc}$ denotes
the $\mc$-component of the composition $T = p \circ \hT$
and $F$ is the composition $F = p \circ \hF$\,.

We claim that $\hT$ is a
$\Ger^{+}_{\infty}$ quasi-isomorphism from
$\Ger^{+}_{\infty}$-algebra $(V^{\bul}_A ,A, \hQ_{\hF})$ to
the $\Ger^{+}_{\infty}$-algebra $(\Cbu(A,A), A, \hQ)$\,.
In other words,
\begin{equation}
\label{hT-Q-QF}
\hQ \circ \hT = \hT \circ \hQ_{\hF}\,.
\end{equation}
Indeed, the $\mc$-component of this equation
$$
\hQ \circ \hT \Big|_{\Ger^{\vee}(V)} =
\hT \circ \hQ_{\hF} \Big|_{\Ger^{\vee}(V)}
$$
holds because $\hG$ is a
$\Ger_{\infty}$-quasi-isomorphism.

To show that the $\ma$-component of (\ref{hT-Q-QF})
holds it suffices to prove that
\begin{equation}
\label{p-hT-Q-QF}
Q^{\ma} \circ \hT^{\ma}
(\ga_1, \dots, \ga_k; a_1, a_2, \dots, a_n)
= T^{\ma} \circ \hQ^{\ma}_{\hF}
(\ga_1, \dots, \ga_k; a_1, a_2, \dots, a_n)\,.
\end{equation}
The latter equation is a consequence of
(\ref{ma-Q0}), (\ref{ma-Q1}), (\ref{ma-Q11}), (\ref{T-ma}),
and (\ref{hF-hT-mc}).

The part ``if'' requires more work.

Let $\hT: \bS(V^{\bul}_A, A) \to \bS(\Cbu(A,A),A)$ be
a morphism of $\bS$-coalgebras compatible with
the codifferentials $\hQ_{\hF}$ on $ \bS(V^{\bul}_A, A)$
and $\hQ$ on $\bS(\Cbu(A,A),A)$\,. As above $T$ is the
composition $p \circ \hT$ with the natural projection
$p : \bS(\Cbu(A,A),A) \to \Cbu(A,A) \oplus A$ and $T^{\mc}$
and $T^{\ma}$ are the $\mc$-component and the
$\ma$-component of $T$, respectively.

Since the $\mc$-component $\hT^{\mc}$
$$
\hT^{\mc} : \Ger^{\vee}(V^{\bul}_A) \to  \Ger^{\vee}(\Cbu(A,A))
$$
is a $\Ger_{\infty}$ quasi-isomorphism
from $V^{\bul}_A$ to the $\Ger_{\infty}$-algebra
$(\Cbu(A,A), \cD)$, we simply need to show that
$\hF$ (\ref{hF}) is homotopic to the
$L_{\infty}$ quasi-isomorphism
\begin{equation}
\label{L-infty-T}
\hT^{\mc}\Big|_{\bs^2 \, \tS(\bs^{-2} V^{\bul}_A)} :
\bs^2 \, \tS(\bs^{-2} V^{\bul}_A) \to
 \bs^2 \, \tS(\bs^{-2} \Cbu(A,A))\,.
\end{equation}

Since $A$ has the zero differential and
it is concentrated in the single degree $0$
we have
\begin{equation}
\label{T-a}
T^{\ma}(a_1) = a_1
\end{equation}
and
\begin{equation}
\label{T-aaa}
T^{\ma}(a_1, \dots, a_n) = 0
\end{equation}
for all $n > 1$ and $a_i \in A$\,.

Unfolding the $\ma$-component of the equation
$\hQ \circ \hT = \hT \circ \hQ_{\hF} $ we get
\begin{equation}
\label{p-hT-Q-QF1}
Q^{\ma} \circ \hT^{\ma}
(\ga_1, \dots, \ga_k; a_1, a_2, \dots, a_n)
=
\end{equation}
$$
\sum_{p =1}^{k}\sum_{\la  \in \Sh_{p,k-p}}
(-1)^{\ve_{\la}} T^{\ma}(Q^{\mc}_{\hF} (\ga_{\la(1)}, \dots,
\ga_{\la(p)}), \ga_{\la(p+1)},
\dots, \ga_{\la(k)} ; a_1, \dots, a_n ) +
$$
$$
\sum_{1\le t \le s \le n}^{0 \le p \le k}
\sum_{\la \in \Sh_{p,k-p}}
(-1)^{\ve_{\la} + \ve^{\la}_{t,p}}
T^{\ma} (\ga_{\la(1)}, \dots, \ga_{\la(p)} ;
a_1, \dots, Q^{\ma}_{\hF}(\ga_{\la(p+1)},
\dots, \ga_{\la(k)}; a_{t}, \dots, a_{s} ) \dots, a_n ) +
$$
$$
\sum_{1 \le t \le n+1}^{0 \le p \le k-1}
\sum_{\la \in \Sh_{p,k-p}}
(-1)^{\ve_{\la} + \ve^\la_{t,p}} T^{\ma} (\ga_{\la(1)}, \dots, \ga_{\la(p)}, ;
a_1, \dots, a_{t-1}, Q^{\ma}_{\hF}(\ga_{\la(p+1)},
\dots, \ga_{\la(k)}), a_{t}, \dots, a_n )
$$
where
\begin{equation}
\label{ve-la-ga}
\ve_{\la} = \sum_{i<j ~ \la(i) > \la(j)}
|\ga_i||\ga_j |\,,
\end{equation}
\begin{equation}
\label{ve-la-tp-ga}
\ve^{\la}_{t,p} = \sum_{j=1}^{p} |\ga_{\la(j)}| +
\sum_{j=p+1}^{k} \sum_{l=1}^{t-1} |\ga_{\la(j)}| (|a_l|+1) +
\sum_{l=1}^{t-1} (|a_l|+1)\,.
\end{equation}

In the case $k=1$ equation (\ref{p-hT-Q-QF1}) reduces to
$$
- T^{\ma}(\ga, a_1, \dots, a_{n-1}) a_n
- (-1)^{|\ga|} a_1 T^{\ma}(\ga, a_2, \dots, a_n)
+ T^{\mc}(\ga)(a_1, \dots, a_n)
=
$$
\begin{equation}
\label{k-equals-1}
F(\ga)(a_1, \dots, a_n) -
(-1)^{|\ga|}  T^{\ma}(\ga, a_1 a_2, a_3, \dots, a_n)
+
(-1)^{|\ga|}  T^{\ma}(\ga, a_1, a_2 a_3, \dots, a_n) + \dots
\end{equation}
$$
+ (-1)^{|\ga| + n-1}  T^{\ma}(\ga, a_1, \dots, a_{n-2}, a_{n-1} a_n)\,.
$$
Here we use (\ref{ma-Q0}), (\ref{ma-Q1}), (\ref{ma-Q11}),
(\ref{T-a}), and (\ref{T-aaa}).

A simple degree bookkeeping shows that if $|\ga|\neq n$
then all terms of equation
(\ref{k-equals-1}) are identically zero.

Thus equation (\ref{k-equals-1}) is equivalent to
\begin{equation}
\label{k-equals-1-equiv}
T^{\mc}(\ga) = F(\ga) + [m, \psi_1(\ga)]_G
\end{equation}
where $m$ is the multiplication in $A$ and
$\psi_1$ is a map from $V^{\bul}_A$ to $\Cbu(A,A)$
of degree $-1$ defined by the formula
$$
\psi_1(\ga)(a_1, a_2, a_3, \dots, a_{|\ga|-1}) =
 T^{\ma}(\ga, a_1, a_2, a_3, \dots, a_{|\ga|-1})\,.
$$

Equation (\ref{k-equals-1-equiv}) implies that
$F$ is an $L_{\infty}$ quasi-isomorphism.
Furthermore, substituting $F$ by a homotopic
 $L_{\infty}$ quasi-isomorphism we can ``kill'' $\psi_1$\,.

In other words, with loss of generality, we may assume that
$\psi_1=0$ and hence
\begin{equation}
\label{T-for-k-equal-1}
 T^{\ma}(\ga, a_1, a_2, a_3, \dots, a_{s}) = 0\,,
\end{equation}
and
\begin{equation}
\label{T-versus-F-1}
T^{\mc}(\ga) = F(\ga)\,.
\end{equation}
for all $\ga\in V^{\bul}_A$ and $a_i\in A$\,.

Let us assume, by induction, that
for all $k' < k$
\begin{equation}
\label{T-for-k-prime}
 T^{\ma}(\ga_1, \dots, \ga_{k'}; a_1, \dots, a_{n}) = 0\,,
\end{equation}
and
\begin{equation}
\label{T-versus-F-k-prime}
T^{\mc}(\ga_1, \dots, \ga_{k'}) = F(\ga_1, \dots, \ga_{k'})
\end{equation}
for all $\ga_i\in V^{\bul}_A$ and $a_j \in A$\,.

Then equation (\ref{p-hT-Q-QF1}) reduces to
$$
 -T^{\ma}(\ga_1, \dots, \ga_k ; a_1, \dots, a_{n-1}) a_n
 - (-1)^{|\ga_1|+ \dots + |\ga_k|}
 a_1 T^{\ma}(\ga_1, \dots, \ga_k ; a_2, \dots, a_n) +
$$
\begin{equation}
\label{k-equals-k}
T^{\mc}(\ga_1, \dots, \ga_k)(a_1, \dots, a_n)  =
F(\ga_1, \dots, \ga_k)(a_1, \dots, a_n) -
\end{equation}
$$
(-1)^{|\ga_1|+ \dots + |\ga_k|}
\big( T^{\ma}(\ga_1, \dots, \ga_k ; a_1 a_2, a_3, \dots, a_n) -
T^{\ma}(\ga_1, \dots, \ga_k ; a_1, a_2 a_3, \dots, a_n) + \dots
$$
$$
 + T^{\ma}(\ga_1, \dots, \ga_k ; a_1, \dots, a_{n-2}, a_{n-1} a_n) \big)\,.
$$
Here we use (\ref{ma-Q0}), (\ref{ma-Q1}), (\ref{ma-Q11}) and
(\ref{T-for-k-prime}).

Again, a simple degree bookkeeping shows that
if $n  \neq |\ga_1|+ \dots + |\ga_k| + 2 - 2k$ then
all terms of equation (\ref{k-equals-k}) are identically
zero.

Thus equation (\ref{k-equals-k}) is equivalent to
\begin{equation}
\label{k-equals-k-equiv}
T^{\mc}(\ga_1, \dots, \ga_k) = F(\ga_1, \dots, \ga_k )
+ [m, \psi_k(\ga_1, \dots, \ga_k )]_G
\end{equation}
where $m$ is the multiplication in $A$ and
$\psi_k$ is a map from $S^{k}(V^{\bul}_A)$ to $\Cbu(A,A)$
of degree $1-2k$ defined by the formula
$$
\psi_k(\ga_1, \dots, \ga_k )(a_1, \dots, a_{n-1}) =
 T^{\ma}(\ga_1, \dots, \ga_k ; a_1, \dots, a_{n-1})
$$
and $n= |\ga_1|+ \dots + |\ga_k| + 2-2k$\,.

Combining this observation with the inductive
assumption (\ref{T-versus-F-k-prime}) we see that
substituting $\hF$ by a homotopic $L_{\infty}$ quasi-isomorphism
we can ``kill''  $\psi_k$ and hence adjust $\hT$ in such
a way that
\begin{equation}
\label{T-for-k}
 T^{\ma}(\ga_1, \dots, \ga_{k}, a_1, \dots, a_{n}) = 0\,,
\end{equation}
and
\begin{equation}
\label{T-versus-F-k}
T^{\mc}(\ga_1, \dots, \ga_{k}) = F(\ga_1, \dots, \ga_{k})
\end{equation}
for all $\ga_i\in V^{\bul}_A$ and $a_j \in A$\,.

This argument completes the step of the induction and
hence the proof of the theorem.

\end{proof}

Let us now restrict our consideration to the case
when $A$ is the polynomial algebra
$$
A=\bbK[x^1, \dots, x^d]
$$
in $d$ variables.
In this case the Gerstenhaber
algebra structure on $V^{\bul}_A$ satisfies the following
rigidity property:
\begin{teo}[Tamarkin's rigidity, \cite{Hinich-pro-Dimu}, \cite{Dima-Proof}]
\label{thm:rigidity}
Let $\bbK$ be a field of characteristic zero.
If $A=\bbK[x^1, x^2, \dots, x^d]$ then any $\Ger_{\infty}$-algebra
structure on $V^{\bul}_A$ with the binary operations
$\wedge$ and $[\,,\,]_{SN}$ is homotopy equivalent to the Gerstenhaber
algebra $(V^{\bul}_A, \wedge, [\,,\,]_{SN})$\,.
\end{teo}
For the proof of this theorem we refer the reader
to Section 5.4 of nice exposition \cite{Hinich-pro-Dimu}
by V. Hinich.

Applying Theorem \ref{thm:transfer}
to the $\Ger^+_{\infty}$-algebra structure
on the pair $(\Cbu(A,A), A)$ we get a transferred $\Ger^+_{\infty}$-algebra
structure on the cohomology $(V^{\bul}_A, A)$\,.

This $\Ger^+_{\infty}$-algebra structure automatically satisfies Property {\bf A}
from Corollary \ref{Ger-plus-gives-F}. Indeed, this fact
follows easily from the observation that $A$ is concentrated
in the single degree $0$\,.

Property {\bf G} from Corollary \ref{Ger-plus-gives-F}
is not satisfied in general. However, using Theorem \ref{thm:rigidity}
we can adjust the transferred
$\Ger^+_{\infty}$-algebra structure $\hQ_H$ on
$(V^{\bul}_A, A)$ so that it will satisfy Property {\bf G}.
It is clear that Property {\bf A} will still be satisfied
for this new $\Ger^+_{\infty}$-algebra structure on $(V^{\bul}_A, A)$\,.

Thus we obtained a $\Ger^+_{\infty}$-algebra structure $\hQ_H$
on the pair $(V^{\bul}_A, A)$ which satisfies Properties
{\bf G} and {\bf A} from Corollary \ref{Ger-plus-gives-F}\,.
In other words, $\hQ_H$ defines (and is defined by)
an $L_{\infty}$-morphism
$$
\hF : V^{\bul}_A \brarrow \Cbu(A,A)\,.
$$
In addition, the $\Ger_{\infty}^+$-algebra $(V^{\bul}_A, A, \hQ_H)$
is quasi-isomorphic to $(\Cbu(A,A), A, \hQ)$\,.
Thus Theorem \ref{Teorema} implies the following
corollary
\begin{cor}
\label{cor:transfer-gives}
Let $A=\bbK[x^1, \dots, x^d]$\,. If
$$
\hF : V^{\bul}_A \brarrow \Cbu(A,A)\,.
$$
is an $L_{\infty}$-morphism obtained, as above, via
transfer of the $\Ger^+_{\infty}$-structure on
$(\Cbu(A,A),A)$ to the cohomology then $\hF$ is
homotopic to an $L_{\infty}$ quasi-isomorphism
constructed via Theorem \ref{thm:nasha}.
\end{cor}

\section{A Link between $\Ger^+_{\infty}$ and Voronov's Swiss Cheese
operad $\SC$}
\label{sect:SC-Ger-plus}

In this section we describe a link between the DG operad
$\Ger^+_{\infty}$ and the Fulton-MacPherson
version $\SC$ of Voronov's Swiss Cheese operad \cite{Sasha-SC}.

Let us first describe the $2$-colored topological operad $\SC$\,.

\subsection{The Swiss Cheese operad $\SC$}
\label{subsect:SC}
To define the spaces of the $2$-colored operad $\SC$ we  introduce 
the configuration space $\Conf_k$ of $k$ labeled distinct
points on the complex plane $\bbC$. Similarly, we introduce
the configuration space $\Conf_{k,n}$  of $k+n$ labeled distinct points  on 
$\bbC$ such that the first $k$ points lie on the 
upper half plane 
\begin{equation}
\label{upper-half}
\bH = \{ z \in \bbC ~|~ {\rm Im} (z)  > 0 \} 
\end{equation}
and the last $n$ points lie 
on the real line $\bbR$.  Next, we define the space $C_k$ as the quotient 
of $\Conf_k$ by the action of the $3$-dimensional real Lie group 
$$
z \mapsto a z + b\,, \qquad \quad a\in \bbR,~ a > 0\,, \quad b \in \bbC \,.    
$$
Similarly, $C_{k,n}$ is the quotient of $\Conf_{k,n}$ by the action 
of the $2$-dimensional real Lie group: 
$$
z \mapsto a z + b\,, \qquad \quad a\in \bbR,~ a > 0\,, \quad b \in \bbR \,.    
$$

The spaces $\SC^{\mc}(k,n)$ are empty whenever $n\ge 1$\,,
\begin{equation}
\label{SC-mc}
\SC^{\mc}(k,0) : = \bC_k\,,
\end{equation}
and
\begin{equation}
\label{SC-ma}
\SC^{\ma}(k,n) : = \bC_{k,n}\,,
\end{equation}
where $\bC_k$ and $\bC_{k,n}$ are respectively compactifications  
of $C_k$ and $C_{k,n}$ constructed by M. Kontsevich in \cite[Section 5]{K}.  

The spaces $\bC_k$ and $\bC_{k,n}$  are obtained 
from $C_k$ and $C_{k,n}$, respectively,  \`a la 
Fulton-MacPherson by performing real blow-ups along all partial 
diagonals.\footnote{Building $\bC_{k,n}$ we should also perform 
real blow-ups along the pieces of the boundary which correspond 
to the situation when a cluster of points in the upper half plane
approaches the real line $\bbR$\,.}

These spaces are  manifolds with corners and
open strata of $\bC_k$ 
are in correspondence with $k$-labeled planar trees. 
More precisely, given a $k$-labeled planar tree $T$
with the set of vertices $V(T)$,
a stratum $C_T$ corresponding to $T$ is 
isomorphic to the following product
\begin{equation}
\label{C-T}
C_T =  \prod_{v \in V(T)\setminus \{{\rm leaves}\}  }  C_{k(v)}\,, 
\end{equation}
where $k(v)$ is the number of incoming\footnote{We tacitly assume that 
all trees are rooted and hence oriented: each vertex has exactly one outgoing 
edge and may have several incoming edges; a vertex with no incoming edges is called 
a leaf.}  edges of the vertex $v$\,.  
The planar tree $T$ tells us in which order we need to perform blow-ups 
to get the stratum $C_T$\,. 
 
The correspondence between planar trees and strata of $\bC_k$ is not bijective. 
For example, if two $k$-labeled planar trees $T$ and $T'$ give the 
same $k$-labeled tree then we identify the strata $C_T$ and $C_{T'}$
by relabeling points in components $C_{k(v)}$, $v \in  V(T)\setminus \{{\rm leaves}\}$
according to the difference in the planar structures of $T$ and $T'$\,.

To describe strata of $\bC_{k,n}$ we need to use 2-colored 
planar trees.  A {\it 2-colored  planar tree} $T$ is a planar 
tree $T$ together with a map 
$$
\chi : E(T) \to \{\mc, \ma\}
$$
which assigns to each edge $e\in E(T)$ a color $\mc$ or 
a color $\ma$\,. We require that the total order $\{\mc  < \ma\}$
on the set $\{\mc, \ma\}$ is compatible with the planar structure
in the following sense: {\it for every vertex $v$ of the tree $T$ the 
restriction of the map $\chi$ to the set $E_{in}(v)$ of edges terminating 
at the vertex $v$ is order preserving}. In other words, if we walk around the vertex $v$ in the
clockwise direction starting with the outgoing edge
then we first cross all the edges in $E_{in}(v)$ with the color  $\mc$  
and then all the edges in $E_{in}(v)$ with the color $\ma$\,.
For an illustration, see the left picture on figure \ref{fig:exam}.
\begin{figure}[htp] 
\centering
\hspace{0.1\linewidth}
\begin{minipage}[t]{0.4\linewidth}
\psset{unit=0.8cm}
\begin{pspicture}(2,2)
\psline[linestyle=dashed](1,0)(1,1)
\psline[linestyle=dashed](1,1)(2,2)
\psline[linestyle=dashed](1,1)(1.5,2)
\psline(1,1)(0,2)
\psline(1,1)(0.5,2)
\psline(1,1)(1,2)
\pscircle*(1,1){0.08}
\end{pspicture} 
\end{minipage} 
\begin{minipage}[t]{0.4\linewidth}
\psset{unit=0.8cm}
\begin{pspicture}(4,4)
\psline[linestyle=dashed](2,0)(2,1)
\psline[linestyle=dashed](2,1)(3,2)
\psline[linestyle=dashed](3,2)(3,3)
\psline[linestyle=dashed](3,2)(3.5,3)
\psline(2,1)(1,2)
\psline(1,2)(0.5,3)
\psline(1,2)(1.5,3)
\psline(3,2)(2.5,3)
\pscircle*(2,0){0.08}
\pscircle*(2,1){0.08}
\pscircle*(1,2){0.08}
\pscircle*(0.5,3){0.08}
\pscircle*(1.5,3){0.08}
\pscircle*(3,2){0.08}
\pscircle*(2.5,3){0.08}
\pscircle*(3,3){0.08}
\pscircle*(3.5,3){0.08}
\rput(0.5,3.4){$3$}
\rput(1.5,3.4){$1$}
\rput(2.5,3.4){$2$}
\rput(3,3.4){$5$}
\rput(3.5,3.4){$4$}
\end{pspicture} 
\end{minipage}
\caption{Dashed edges have the color $\ma$, solid edges have the color $\mc$}  
\label{fig:exam}
\end{figure}
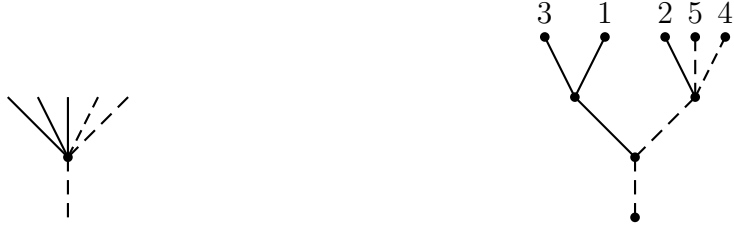

For our purposes we also impose the following 
condition on a 2-colored  planar tree $T$: 
\begin{cond}
\label{cond:ma-ma}
If at least one incoming edge of a vertex $v$ has the color $\ma$ then 
its outgoing edge has the color $\ma$.
\end{cond}

For every 2-colored planar tree $T$ the set $V(T)$ of vertices 
splits into the disjoint union according to the 
color of the outgoing edge 
$$
V(T) = V_{\mc}(T) \sqcup V_{\ma}(T)\,.  
$$
In other words, 
a vertex $v$ belongs to $V_{\mc}(T)$ (resp. $V_{\ma}(T)$) if 
its outgoing edge has the color $\mc$ (resp. the color $\ma$).
The above condition implies that for every vertex $v\in V_{\mc}(T)$
all its incoming edges have the color $\mc$\,.

A $(k,n)$-labeled 2-colored planar tree $T$  
is a 2-colored planar tree $T$ equipped with a bijection between 
the set of leaves and the set $\{1,2, \dots, k+n\}$ such that 
the external edges corresponding to the first $k$ leaves have 
the color $\mc$ and the external edges corresponding to the 
last $n$ leaves have  the color $\ma$\,.    
For an example of a $(3,2)$-labeled 2-colored planar tree see 
the right picture on figure \ref{fig:exam}.  

We can now describe open strata of $\bC_{k,n}$\,.
A stratum $C_T$ of $\bC_{k,n}$ 
corresponding to a $(k,n)$-labeled 2-colored planar tree $T$  
is the following product
\begin{equation}
\label{C-T-ma}
C_T =  \prod_{v \in V_{\mc}(T)\setminus \{{\rm leaves}\}  }  C_{k(v)} 
~~ \times ~~   \prod_{v \in V_{\ma}(T)\setminus \{{\rm leaves}\}  }  C_{k(v), n(v)}\,, 
\end{equation}
where $k(v)$ (resp. $n(v)$) denotes the number of incoming edges with color $\mc$
(resp. with color $\ma$)\,. 

\begin{remark}
\label{rem:amendment}
Condition \ref{cond:ma-ma} implies that the root 
edge of every  $(k,n)$-labeled 2-colored planar 
tree $T$ with $n\neq 0$
has necessarily the color $\ma$\,. 
In the case $n=0$ we require, in addition, that 
the root of $T$ in (\ref{C-T-ma})
has the color $\ma$\,. For example, 
we may have a $(3,0)$-labeled 2-colored 
planar tree depicted on figure \ref{fig:possible}.
\begin{figure}[htp] 
\centering 
\psset{unit=0.8cm}
\begin{pspicture}(2,3)
\psline[linestyle=dashed](1,0)(1,1)
\psline(1,1)(2,2)
\psline(1,1)(1,2)
\psline(1,1)(0,2)
\pscircle*(1,0){0.08}
\pscircle*(1,1){0.08}
\pscircle*(2,2){0.08}
\pscircle*(1,2){0.08}
\pscircle*(0,2){0.08}
\rput(0,2.4){$1$} 
\rput(1,2.4){$3$}
\rput(2,2.4){$2$}
\end{pspicture}
\caption{The dashed edge has the color $\ma$, solid edges have the color $\mc$}  
\label{fig:possible}
\end{figure}
\end{remark}

As well as for $\bC_k$
the correspondence between planar 2-colored trees and strata 
of $\bC_{k,n}$ is not bijective. 
For example, if two $(k,n)$-labeled 2-colored planar trees $T$ and $T'$ give the 
same $(k,n)$-labeled 2-colored tree then we identify the strata $C_T$ and $C_{T'}$
by relabeling points in components $C_{k(v)}$, $v \in  V_{\mc}(T)\setminus \{{\rm leaves}\}$
and $C_{k(v), n(v)}$, $v \in  V_{\ma}(T)\setminus \{{\rm leaves}\}$
according to the difference in the planar structures of $T$ and $T'$\,.

The operadic insertions for $\SC$ are defined in the obvious way 
using concatenation of trees and identification of the corresponding strata. 

Following this description, it is not hard to see that, in the category of sets, 
$\SC$  is the free $2$-colored operad 
generated by the collection $\{C_k, C_{k,n}\}$ where 
$C_k$'s give us operations of arity
$$
(\underbrace{\mc, \dots, \mc}_{k} ~ \mapsto ~ \mc )
$$
 and $C_{k,n}$'s give us operations 
of arity
$$
(\underbrace{\mc, \dots, \mc}_{k}, 
\underbrace{\ma, \dots, \ma}_{n} ~ \mapsto ~ \ma )\,.
$$ 
 
\subsection{The operad  $H_{\bul}(\SC)$ and the cooperads $H^{\bul}(\SC)$ and $\sc$}
\label{subsect:S-versus-sc}
The homology operad $H_{\bul}(\SC)$ was described in 
\cite{Sasha-SC}: 
\begin{teo}[A. Voronov, \cite{Sasha-SC}]
Algebras over $H_{\bul}(\SC)$ are triples $(V,A, \vr)$, 
where $V$ is a Gerstenhaber algebra, $A$ is an associative 
algebra and  
$$
\vr : V \to A
$$ 
is a homomorphism of associative algebras satisfying 
the condition 
\begin{equation}
\label{vr-center}
\vr(v) a= (-1)^{|v| |a|} a \vr(v)
\end{equation}
for all $v\in V$ and $a\in A$\,.
\end{teo}
\begin{proof}
Although a proof is given in \cite{Sasha-SC} we would like to briefly 
explain how one can prove this statement without using the original 
(non-Fulton-MacPherson) version of the Swiss Cheese operad. 

Spaces of the operad $\SC$ are the compatifications 
$\bC_{k}$ and $\bC_{k,n}$ described above.  
The top dimensional stratum of $\bC_k$ (resp. $\bC_{k,n}$)
is the configuration space $C_k$ (resp. $C_{k,n}$). Thus, by the 
collar neighborhood theorem, one can construct a homotopy 
equivalence between $\bC_k$  (resp. $\bC_{k,n}$) and 
$C_k$ (resp. $C_{k,n}$). Hence 
\begin{equation}
\label{homology-bC-C}
H_{\bul}(\bC_k) \cong H_{\bul}(C_k)\,, \qquad 
H_{\bul}(\bC_{k,n}) \cong H_{\bul}(C_{k,n})\,.
\end{equation}
The homology of the spaces $C_k$ and $C_{k,n}$ 
can be computed along the lines of \cite{Cohen} or 
\cite{GJ}. Finally, using the homotopy equivalence
between  $\bC_k$  (resp. $\bC_{k,n}$) and 
$C_k$ (resp. $C_{k,n}$) it is not hard to see how 
the desired operad assembles from the
homology spaces $H_{\bul}(C_{k})$ and $H_{\bul}(C_{k,n})$\,.  
\end{proof}
\begin{remark}
\label{rem:cohomology-bC-C}
Using the collar neighborhood theorem again 
(or applying the universal coefficient theorem) we 
get the similar description for the spaces of 
the linear dual cooperad  $H^{\bul}(\SC)$:  
\begin{equation}
\label{cohomology-bC-C}
H^{\bul}(\bC_k) \cong H^{\bul}(C_k) \qquad 
H^{\bul}(\bC_{k,n}) \cong H^{\bul}(C_{k,n})\,.
\end{equation}
\end{remark}

The cooperad $H^{\bul}(\SC)$ is closely 
related to the cooperad $\bS$ (see Definition \ref{dfn:S-coalg}). 
To describe a link between these two cooperads  
we introduce the following modification 
of $H^{\bul}(\SC)$-coalgebras:
\begin{defi}
\label{dfn:sc-coalg}
An $\sc$-coalgebra is a pair of $\Ger^{\vee}$-coalgebra
$V$ and a $\La\coass$-coalgebra $A$ together with a unary
operation
$$
\rho : A \to V
$$
of degree\footnote{Recall that by degree of $\rho$ we mean the degree 
of the corresponding vector in the cooperad governing $\sc$-coalgebras.}  
$-1$ such that the following diagrams commute:
\begin{equation}
\label{rho-map-sc-coalg}
\xymatrix@M=0.4pc{
A\ar[r]^{\rho} \ar[d]_{\D_{\ma}}  &  V \ar[d]^{\D_{\mc}}  \\
A\otimes A \ar[r]^{\rho \otimes \rho} & V \otimes V
}
\end{equation}
\begin{equation}
\label{sc-left-right-via-rho}
\xymatrix@M=0.4pc{
~ & A \otimes A \ar[r]^{\rho \otimes \id} & V\otimes A\\
A \ar[ur]^{\D_{\ma}} \ar[dr]_{\D_{\ma}} & ~ & ~ \\
~ & A \otimes A \ar[r]^{\id \otimes \rho} & A\otimes V \ar[uu]^{\si}
}
\end{equation}
where $\D_{\mc}$ (resp. $\D_{\ma}$) is the comultiplication in $V$
(resp. in $A$) and
$$
\si: A \otimes V \to V \otimes A
$$
is the mapping which switches
tensor components. 
We will denote by $\sc$ the $2$-colored cooperad which
governs $\sc$-coalgebras. 
\end{defi}
\begin{remark}
\label{rem:sc-versus-H-SC}
It is clear that the cooperad $\sc$ is obtained 
from  $H^{\bul}(\SC)$ by shifting the degrees of the 
cooperations $\D_{\mc}$, $\de_{\mc}$, 
$\rho$, $\D_{\ma}$ by 
$-2$, $-2$, $-1$, and $-1$, respectively. 
In other words, we have
\begin{equation}
\label{eq:sc-versus-H-SC}
\sc^{\mc}(k,0) = \bs^{2-2k} H^{\bul}(\SC)^{\mc}(k,0)\,, 
\qquad 
\sc^{\ma}(k,n) = \bs^{1-n-2k} \, (1\otimes \sgn_n) \otimes  H^{\bul}(\SC)^{\ma}(k,n)\,, 
 \end{equation}
where $1 \otimes \sgn_n$ is the product of the trivial representation 
of $S_k$ and the sign representation of $S_n$\,.
\end{remark}
It is not hard to see that 
an $\bS$-coalgebra $(V,A)$ is an $\sc$-coalgebra
with an additional condition:
$$
\de_{\mc} \circ \rho = 0\,,
$$
where $\de_{\mc}$ denotes the Lie cobracket on $V$\,.
Thus the cooperad $\bS$ is a sub-cooperad of $\sc$
\begin{equation}
\label{bS-in-sc}
\bS \hookrightarrow \sc\,.
\end{equation}

\subsection{The first sheet $E^1(\SC)$ of the spectral sequence for $\SC$}
Using the stratification of the spaces 
(\ref{SC-mc}), (\ref{SC-ma}) described in Subsection \ref{subsect:SC}
we equip the topological operad $\SC$ with the increasing filtration
\begin{equation}
\label{SC-filtr}
\dots \subset F^p\,\SC \subset F^{p+1}\,\SC \subset \dots
\end{equation}
where $F^p\,\SC$ is the closure of the union of all strata
of dimension $p$\,.

In \cite{Sasha-SC} A. Voronov considered the first sheet 
$E^1(\SC)$ of the spectral sequence corresponding to 
this filtration. He referred to the DG operad $E^1(\SC)$
as the operad governing ``homotopy Swiss Cheese algebras''.
To build a link between the DG operads
$\Ger^+_{\infty}$ and $E^1(\SC)$ we would like 
to show that
\begin{teo}
\label{thm:E-1-SC}
The DG operads $E^1(\SC)$ and $\Cobar(\sc)$
are isomorphic.
\end{teo}
\begin{proof}
The spaces $\bC_k$ form a Fulton-MacPherson version $\FME_2$
of the little disc operad. So the operad  $\SC$ is a $2$-colored 
extension of the ordinary (non-colored) operad $\FME_2$\,. 

Thus if we restrict our consideration to operations 
of  $E^1(\SC)$  of the arity
$$
(\mc,\mc, \dots, \mc ~ \mapsto ~ \mc )
$$
then we get an ordinary (non-colored) DG operad 
\begin{equation}
\label{ccc-c}
E^1(\FME_2)\,.
\end{equation}

In the famous Getzler-Jones paper\footnote{See also 
Proposition 1.2 in \cite{Sasha}.}  \cite{GJ} the DG operad 
(\ref{ccc-c}) was identified with the cobar construction 
$\Cobar(\Ger^{\vee})$ of the cooperad $\Ger^{\vee}$\,.

Thus, to prove the theorem, it remains to extend this 
identification to the operations of the arity 
$$
(\underbrace{\mc, \dots, \mc}_{k}, 
\underbrace{\ma, \dots, \ma}_{n} ~ \mapsto ~ \ma )
$$

For $E^1_{\bul, \bul}$ of $\bC_{k,n}$ we have\footnote{Recall that 
we use the reversed grading on all homological complexes.}
\begin{equation}
\label{E-1-bC}
E^1_{p,q}(\bC_{k,n}) =
H_{p + q}(F^{-p}\,\bC_{k,n}, F^{-p-1}\, \bC_{k,n} )\,.
\end{equation}

$F^{-p-1}\bC_{k,n}$ is the boundary of the compact manifold 
$F^{-p}\bC_{k,n}$\,. Hence by the Poincar\'e-Lefschetz duality we have
\begin{equation}
\label{Poincare}
H_{p+q}(F^{-p}\,\bC_{k,n}, F^{-p-1}\,\bC_{k,n} )
\cong H^q (F^{-p}\, \bC_{k,n} )\,.
\end{equation}

Next, the collar neighborhood theorem implies that 
the space $F^{-p}\, \bC_{k,n} $ is homotopy equivalent
to its interior  
\begin{equation}
\label{F-p-interior}
F^{-p}\,\bC_{k,n} \, \setminus \, F^{-p-1}\, \bC_{k,n}\,.
\end{equation}
Therefore, we have the isomorphism 
\begin{equation}
\label{collar}
H^q (F^{-p}\, \bC_{k,n} ) \cong
H^q (F^{-p}\,\bC_{k,n} \, \setminus \, F^{-p-1}\, \bC_{k,n})\,.
\end{equation}

The space (\ref{F-p-interior})
is the disjoint union of strata of dimension $-p$\,.
Hence, using (\ref{C-T-ma}) and the K\"unneth formula  
we get 
$$
H_{p+q}(F^{-p}\,\bC_{k,n}, F^{-p-1}\,\bC_{k,n}  )
\cong 
$$
\begin{equation}
\label{unfolding}
\bigoplus_T
\left( 
\bigotimes_{v \in V_{\mc} (T)\setminus \{{\rm leaves}\}} H^{\bul}(C_{k(v)}) 
~~\otimes ~~
\bigotimes_{v\in V_{\ma}(T)\setminus \{{\rm leaves}\}} H^{\bul} (C_{k(v), n(v)})
\right)^q \Big/ ~\sim~
\end{equation}
where  the summation runs over all $(k,n)$-labeled $2$-colored planar trees $T$ and
$k(v)$ (resp. $n(v)$) denotes the number of incoming edges of 
the vertex $v$ with the color $\mc$
(resp. with the color $\ma$). In the right hand side we identify two vectors
if one of them is obtained from another by  
changing the total order on the set of the incoming edges of a vertex 
$v \in V_{\mc}(T)\setminus \{{\rm leaves}\}$ 
(resp. $v \in V_{\ma}(T)\setminus \{{\rm leaves}\}$)
and applying the corresponding 
element of the symmetric group to the component in $H^{\bul}(C_{k(v)})$ 
(resp.  $H^{\bul}(C_{k(v), n(v)})$).

Since the dimension of the stratum $C_T$ corresponding to the tree $T$
is 
$$
-p = dim (C_T) = \sum_{v \in V_{\mc}(T)} (2 k_v-3) + \sum_{v \in V_{\ma}(T)} (2k_v + n_v -2)  
$$
we get that the total degree of a vector in (\ref{unfolding}) reads: 
\begin{equation}
\label{total-deg}
p+q =  q  +  \sum_{v\in V_{\mc}(T)} (2 - 2 k_v) + \sum_{v \in V_{\ma}(T)} (1-2k_v - n_v)
+ |V_{\mc}(T)| +   |V_{\ma}(T)|\,.
\end{equation} 
Combining this computation with (\ref{cohomology-bC-C}) and Remark  
 \ref{rem:sc-versus-H-SC} we get the desired identification 
between $E^1(\SC)$ and $\Cobar(\sc)$ as operads of graded vector spaces. 

To show that the differential $d_1$ in $E^1(\SC)$ coincides with the 
differential $\pa^{\Cobar}$ of the cobar construction 
we need to prove commutativity of the following diagram
\begin{equation}
\label{diag-d-1-and-cobar}
\xymatrix@M=0.15pc{
H^q(F^{-p}\bC_{k,n} \setminus F^{-p-1}\bC_{k,n}) \ar[d]^{\pa^{\Cobar}} & 
H^q(F^{-p}\bC_{k,n})\ar[l] \ar[d]^{j^*}\ar[r] &  
H_{p+q}(F^{-p}\bC_{k,n}, F^{-p-1}\bC_{k,n})\ar[d]^{d_1}\\
H^q(F^{-p-1} \bC_{k,n} \setminus F^{-p-2} \bC_{k,n})  &  H^q(F^{-p-1}\bC_{k,n})\ar[l] \ar[r]  &  
H_{p+q+1}(F^{-p-1}\bC_{k,n}, F^{-p-2}\bC_{k,n}) 
}
\end{equation}   
Here $j$ is the embedding $j: F^{-p-1}\bC_{k,n} \hookrightarrow F^{-p}\bC_{k,n} $, the horizontal 
arrows in the right hand side are the cap-products with $[F^{-p}\bC_{k,n}]$ and  $[F^{-p-1}\bC_{k,n}]$, 
respectively. Finally, the horizontal arrows in the left hand side are isomorphisms 
built using the collar neighborhood theorem. 
 
Let us recall that 
the differential $d_1$ in $E^1(\SC)$ is the composition 
of the coboundary operator 
\begin{equation}
\label{coboundary}
H_{p+q} (F^{-p} \bC_{k,n}, F^{-p-1} \bC_{k,n}) \to 
H_{p+ q +1}( F^{-p-1} \bC_{k,n}) 
\end{equation}   
with the natural map 
\begin{equation}
\label{projection}
H_{p+q+1}( F^{-p-1} \bC_{k,n}) \to H_{p+q+1}( F^{-p-1} \bC_{k,n}, F^{-p-2} \bC_{k,n})  
\end{equation}
Thus the commutativity of the right half of the diagram follows 
from the naturality of the cap-product.

To show that the left half of the diagram (\ref{diag-d-1-and-cobar}) commutes
we let $C_T$ be a top dimensional stratum of $F^{-p}\bC_{k,n}$ corresponding 
to a 2-colored tree $T$\,. Then each top dimensional stratum $C_{\hT}$ of the boundary 
$\pa C_T$ corresponds to a 2-colored tree $\hT$ such that $T$ is obtained from
$\hT$ by contracting exactly one edge of $\hT$\,.  
 Using this observation it is not hard to see that the map $j^*$  is 
expressed in terms of cooperadic coinsertions for 
$H^{\bul}(\SC)$\,.   
 
Theorem \ref{thm:E-1-SC} is proved. 
 
\end{proof}

The embedding of cooperads (\ref{bS-in-sc}) gives us the 
embedding of the DG operads: 
\begin{equation}
\label{Ger-plus-in-E1}
\Ger^+_{\infty} = \Cobar(\bS) \hookrightarrow
\Cobar(\sc) \cong E^1(\SC)\,.
\end{equation}
Thus $\Ger^+_{\infty}$ is a sub DG operad of $E^1(\SC)$\,. 

Combining  (\ref{Ger-plus-in-E1}) with (\ref{OC-inf-in-Ger-plus})
we get the following pair of embeddings of DG operads: 
\begin{equation}
\label{pair-of-embeddings}
\OC_{\infty}  \hookrightarrow
\Ger^+_{\infty}  \hookrightarrow E^1(\SC)\,.
\end{equation}
The embedding $\OC_{\infty}  \hookrightarrow E^1(\SC)$ was 
described in details in paper \cite{OCHA} by E. Hoefel.

\section{The DG operads $\Ger^+_{\infty}$ and $E^1(\SC)$ are
non-formal}
\label{sect:non-formal}
Let us recall that the DG operad $\Ger^+_{\infty}$ is the cobar construction
$\Cobar(\bS)$
of the cooperad $\bS$\,.   The minimal model
condition (\ref{minimal}) implies that $\Ger^+_{\infty}$
is cofibrant. Hence, if the operad $\Ger^+_{\infty}$ is
formal then there exists a quasi-isomorphism of DG operads
\begin{equation}
\label{does-not-exist}
\vf : \Cobar(\bS) \to H
\end{equation}
where $H=H^{\bul}(\Cobar(\bS))$\,.

Let us consider the element
\begin{equation}
\label{X}
X = \bs\, (\rho \otimes \id) \D_{\ma} = \bs\, (\id \otimes \rho) \D_{\ma}  
\end{equation}
in $\Cobar(\bS)^{\ma}(1,1)$\,.

Since $\rho$ and $\D_{\ma}$ have the degree $-1$, the 
element $X$ has the degree $-1$.    
It is not hard to see that $X$ spans the subspace of
all operations having the arity $((\mc, \ma) \to \ma)$ and
the degree $-1$\,. On the other hand, $X$ 
is not a cocycle:
$$
\pa^{\Cobar} ( X ) 
=   \bs \D_{\ma}~ \circ_1~ \bs \rho  +  \bs \D_{\ma} ~\circ_2~ \bs \rho   \neq 0\,,
$$
where $\circ_1$ (resp. $\circ_2$) denotes the operadic insertion in the 
first slot (resp. the second slot). 

Therefore, in $H$ there is no non-zero vector having the
arity $((\mc, \ma) \to \ma)$ and degree $-1$\,.
Hence, we have
\begin{equation}
\label{X-go-to-zero}
\vf(X) =0\,.
\end{equation}

Let $Z$ be the following element of $\Cobar(\bS)$ 
\begin{equation}
\label{Z}
Z = \bs \, (\rho \otimes \rho) \D_{\ma}\,.
\end{equation}
Due to axioms of the $\bS$-algebra the element $Z$ can 
be also written as
$$
Z = \bs \, \D_{\mc} \rho \,.
$$

Next we compute  $\pa^{\Cobar}(Z)$ 
\begin{equation}
\label{pa-Z}
\pa^{\Cobar}(Z) = X ~\circ_2~ \bs \rho - X ~\circ_1~ \bs \rho 
- Y\,,
\end{equation}
where
\begin{equation}
\label{Y} 
 Y=  \bs \, \rho  ~\circ_1~ \bs\, \D_{\mc}\,.
\end{equation}

The element $Y$ is clearly a non-trivial cocycle in $\Cobar(\bS)$\,.
On the other hand, equations (\ref{X-go-to-zero}) and (\ref{pa-Z}) imply that
\begin{equation}
\label{Y-goes-to-zero}
\vf(Y) = 0\,.
\end{equation}
Using this contradiction we deduce that
the DG operad $\Ger^+_{\infty}$ is non-formal.
In other words
\begin{teo}
\label{thm:Ger-plus-non-formal}
It is not possible to construct a sequence of
quasi-isomorphisms of DG operads which connects
$\Ger^+_{\infty}$ to its cohomology.
\end{teo}

Using the identical line of arguments it is
easy to prove that
\begin{teo}
\label{thm:E-1-SC-non-formal}
It is not possible to construct a sequence of
quasi-isomorphisms of DG operads which connects
$E^1(\SC)$ to its cohomology.
\end{teo}

\section{Discussion}
It is known \cite{Kevin}, \cite{Mnev} that path integrals can
be used to transfer homotopy algebraic structures.
Following this line, the computations of A.S. Cattaneo and G. Felder in 
\cite{CF} can be considered as an example of the transfer of the tautological OCHA
on $(\Cbu(A,A), A)$ from Example \ref{OCHA-Cbu-A}
to its cohomology $(HH^{\bul}(A,A), A)$ for $A= \bbR[x^1, \dots, x^d]$\,.   
  
It would be interesting to check whether the Cattaneo-Felder construction
\cite{CF} can be extended further to provide us with a transfer of the
$\Ger^+_{\infty}$-algebra structure on $(\Cbu(A,A), A)$
to the cohomology. If this is the case then Kontsevich's 
formality quasi-isomorphism \cite{K} is 
homotopic\footnote{As far as the author knows, T. Willwacher \cite{Thomas} 
proved that Kontsevich's formality quasi-isomorphism is indeed 
homotopic to the one obtained via Tamarkin's procedure. } 
to 
the one obtained via Tamarkin's procedure \cite{DTT}, \cite{Hinich-pro-Dimu}, 
\cite{Dima-Proof}.

It would be also interesting to answer these questions:

\begin{itemize}

\item What is the role of the full operad $E^1(\SC)$ in
the construction of formality quasi-isomorphisms?

\item Is there a relation between the $\Ger^{+}_{\infty}$-algebra
structure on $(\Cbu(A,A), A)$ introduced in Theorem \ref{Ger-plus-on-C-A}
and the solution of the Swiss Cheese conjecture proposed 
in \cite{SwissCheese}?

\end{itemize}

As for the first question, 
it is natural to guess that $E^1(\SC)$-algebra structure 
on a pair $(V,A)$ is the following data: 
\begin{enumerate}

\item A $\Ger_{\infty}$-structure on $V$\,,

\item an $A_{\infty}$-structure on $A$\,, and 

\item a $\Ger_{\infty}$-morphism from $V$ to $\Cbu(A,A)$\,,

\end{enumerate}
where $\Cbu(A,A)$ carries Tamarkin's $\Ger_{\infty}$-structure. 

This naive guess turns out to be wrong. However, by correcting 
the differential $d_1$ on $E^1(\SC)$ one 
can probably get algebraic structures described 
by the above data. Unfortunately, the author does not 
know any topological meaning of such corrections.

We suspect that the answer to the first question may 
shed some light on the second one.

~\\

\noindent\textsc{Department of Mathematics,
Temple University, \\
1805 N. Broad St.,\\
Philadelphia PA, 19122 USA \\
\emph{E-mail address:} {\bf vald@temple.edu}}


\begin{thebibliography}{99}

\bibitem{BF} C. Berger and B. Fresse,
Combinatorial operad actions on cochains,
Math. Proc. Cambridge Philos. Soc.,
{\bf 137}, 1  (2004) 135--174.


\bibitem{BM} C. Berger and I. Moerdijk, Axiomatic homotopy theory
for operads, Comment. Math. Helv. {\bf 78}, 4 (2003) 805--831;
arXiv:math/0206094.

\bibitem{BM-colors} C. Berger and I. Moerdijk, Resolution of coloured operads and 
rectification of homotopy algebras, Contemp. Math. {\bf 431} (2007) 31--58; 
arXiv:math/0512576.

\bibitem{CF} A.S. Cattaneo and G. Felder,
A path integral approach to the Kontsevich quantization formula,
Commun. Math. Phys. {\bf 212} (2000) 591--611;
arXiv:math/9902090.

\bibitem{ChuangLazarev} J. Chuang and A. Lazarev,
Feynman diagrams and minimal models for operadic algebras,
arXiv:0802.3507.

\bibitem{Cohen} F. R. Cohen, The homology of
$\cC_{n+1}$-spaces, $n\ge 0$\,, {\it The homology of iterated loop
spaces}, Springer-Verlag, 1976, Lecture Notes Math.,
{\bf 533}, 207--351.



\bibitem{Kevin} K.J. Costello, Renormalisation and the
Batalin-Vilkovisky formalism, arXiv:0706.1533.


\bibitem{Erratum} V. Dolgushev,
Erratum to: "A Proof of Tsygan's Formality Conjecture for
an Arbitrary Smooth Manifold", arXiv:math/0703113.

\bibitem{SwissCheese} V. Dolgushev, D. Tamarkin, and B. Tsygan,
Proof of Swiss Cheese Version of DeligneÕs Conjecture, 
accepted to IMRN; arXiv:0904.2753.

\bibitem{DTT}  V.A. Dolgushev, D.E. Tamarkin, and B.L. Tsygan,
The homotopy Gerstenhaber algebra of Hochschild cochains of a
regular algebra is formal, J. Noncomm. Geom. {\bf 1}, 1
(2007) 1-25; math.KT/0605141.


\bibitem{Drinfeld} V.G. Drinfeld, On quasitriangular quasi-Hopf algebras and
on a group that is closely connected with ${\rm Gal}(\overline{\bbQ}/\bbQ)$.
(Russian) Algebra i Analiz {\bf 2}, 4 (1990) 149--181;
translation in Leningrad Math. J. {\bf 2}, 4 (1991) 829--860.


\bibitem{EK} P. Etingof and D. Kazhdan, Quantization of Lie bialgebras.
II, III, Selecta Math. (N.S.) {\bf 4}, 2  (1998) 213--231,
233--269.


\bibitem{Fresse} B. Fresse,
Koszul duality of operads and homology of partition posets,
in "Homotopy theory and its applications
(Evanston, 2002)", Contemp. Math. {\bf 346} (2004)
115--215.


\bibitem{Ger} M. Gerstenhaber, The cohomology structure of
an associative ring, Ann. Math. {\bf 78} (1963) 267--288.

\bibitem{GerVor} M. Gerstenhaber and A.A. Voronov, 
Homotopy G-algebras and moduli space operad,
Internat. Math. Res. Notices, 3 (1995) 141--153.

\bibitem{Ezra} E. Getzler, Cartan homotopy formulas and
the Gauss-Manin connection in cyclic homology, in
{\it Quantum deformations of algebras and their representations},
Israel Math. Conf. Proc. {\bf 7} (1993) 65--78.


\bibitem{Ezra-higher} E. Getzler, A Darboux theorem for Hamiltonian
operators in the formal calculus of variations,
Duke Math. J. {\bf 111}, 3 (2002) 535--560.

\bibitem{Ezra-Lie} E. Getzler,
Lie theory for nilpotent L-infinity algebras,
Ann. of Math. (2) {\bf 170}, 1 (2009) 271--301;
arXiv:math/0404003.


\bibitem{GJ} E. Getzler and J.D.S. Jones,
Operads, homotopy algebra and iterated integrals
for double loop spaces, hep-th/9403055.


\bibitem{GK} V. Ginzburg and M. Kapranov,
Koszul duality for operads,  Duke Math. J.
{\bf 76}, 1 (1994) 203--272.


\bibitem{Hinich-pro-Dimu} V. Hinich, Tamarkin's proof of
Kontsevich formality theorem, Forum Math. {\bf 15},
4 (2003) 591--614; math.QA/0003052.

\bibitem{HKR} G. Hochschild, B. Kostant, and A. Rosenberg,
Differential forms on regular affine algebras,
Trans. Amer. Math. Soc. {\bf 102} (1962) 383--408.

\bibitem{OCHA} E. Hoefel, OCHA and the swiss-cheese operad, 
J. Homotopy Relat. Struct. {\bf 4}, 1 (2009) 123--151. 

\bibitem{OCHA1} E. Hoefel, On the coalgebra description of OCHA, 
arXiv:math/0607435. 

\bibitem{Kadeishvili}  T.V. Kadeishvili, The algebraic structure in the homology 
of an $A_{\infty}$-algebra. (Russian)  Soobshch. Akad. Nauk Gruzin. SSR  {\bf 108}, 2  
(1982) 249--252.

\bibitem{braces}  T.V. Kadeishvili, The structure of the $A(\infty)$-algebra,
and the Hochschild and Harrison cohomologies. (Russian)
Trudy Tbiliss. Mat. Inst. Razmadze. Akad. Nauk Gruzin. SSR
{\bf 91}  (1988) 19--27.

\bibitem{OCHAStasheff}  H. Kajiura and J. Stasheff, Homotopy algebras inspired by 
classical open-closed string field theory, Commun. Math. Phys. {\bf 263} (2006) 553--581; 
arXiv:math/0410291.

\bibitem{OCHAStasheff1}  H. Kajiura and J. Stasheff,  Open-closed homotopy algebra in 
mathematical physics,  J. Math. Phys. {\bf 47}, 2  (2006) 023506, 28 pp; arXiv:hep-th/0510118.


\bibitem{OCHAStasheff11}  H. Kajiura and J. Stasheff,  Homotopy algebra of 
open-closed strings,  {\it Groups, homotopy and configuration spaces}, 229--259, 
Geom. Topol. Monogr., {\bf 13}, Geom. Topol. Publ., Coventry, 2008; arXiv:hep-th/0606283.




\bibitem{K} M. Kontsevich, Deformation quantization of
Poisson manifolds, Lett. Math. Phys.,
{\bf 66} (2003) 157-216; q-alg/9709040.

\bibitem{K1} M. Kontsevich, Operads and motives in
deformation quantization, Lett. Math. Phys.
{\bf 48} (1999) 35--72; math.QA/9904055.


\bibitem{KS} M. Kontsevich and Y. Soibelman,
Deformations of algebras over operads and Deligne's conjecture,
{\it Conf\'erence Mosh\'e Flato 1999, Vol. I (Dijon),}  255--307,
Math. Phys. Stud., 21, Kluwer Acad. Publ., Dordrecht, 2000;
math.QA/0001151.


\bibitem{Laan} P. van der Laan, Coloured Koszul duality and
strongly homotopy operads, arXiv:math/0312147.

\bibitem{Volic} P. Lambrechts and I. Volic,
Formality of the little N-disks operad, arXiv:0808.0457.

\bibitem{M-Smith} J. E. McClure and J. H. Smith,
A solution of Deligne's Hochschild cohomology conjecture,
Contemp. Math. {\bf 293} (2002) 153--193,
Amer. Math. Soc., Providence, RI;
math.QA/9910126.

\bibitem{Mnev} P. Mnev, Notes on simplicial BF theory,
Mosc. Math. J.  {\bf 9},  2  (2009) 371--410; arXiv:hep-th/0610326.

\bibitem{Shoikhet} B. Shoikhet, An explicit construction of the
Quillen homotopical category of dg Lie algebras, arXiv:0706.1333.

\bibitem{Stasheff} J. Stasheff, Homotopy associativity of H-spaces. I, II, 
Trans. Amer. Math. Soc.  {\bf 108} (1963) 275--292; ibid. {\bf 108} (1963) 293--312. 

\bibitem{Sullivan} D. Sullivan, Infinitesimal computations in topology,
IHES. Publ. Math. No. {\bf 47} (1977) 269--331.

\bibitem{Dima-Proof} D. Tamarkin,
Another proof of M. Kontsevich formality theorem,
math.QA/9803025.


\bibitem{Dima-Disk} D. Tamarkin,
Formality of chain operad of small squares,
Lett. Math. Phys. {\bf 66}, 1-2 (2003) 65--72;
math.QA/9809164.


\bibitem{Dima-GT} D. Tamarkin, Action of the
Grothendieck-Teichm\"uller group
on the operad of Gerstenhaber algebras,
arXiv:math/0202039.


\bibitem{Dima-DG} D. Tamarkin, What do DG categories
form?  Compos. Math. {\bf 143}, 5 (2007) 1335--1358;
math.CT/0606553.

\bibitem{TT} D. Tamarkin and B. Tsygan,
Non-commutative differential calculus, homotopy BV algebras and
formality conjectures, Methods Funct. Anal. Topology {\bf 6}, 2 (2000)
85--100; math.KT/0002116.

\bibitem{Sasha}  A.A. Voronov, Homotopy Gerstenhaber
algebras, Proceedings of the Mosh\'{e} Flato conference,
Kluwer Academic Publishers, the Netherlands, {\bf 2}
(2000) 307--331; math.QA/9908040.

\bibitem{Sasha-SC} A.A. Voronov, The Swiss-Cheese Operad, Proc. of the
conference: ``Homotopy invariant algebraic structures''
(Baltimore, MD, 1998),  365--373, {\it Contemp. Math.},
239, Amer. Math. Soc., Providence, RI, 1999.

\bibitem{Thomas} T. Willwacher, private communication. 

\bibitem{Zwiebach} B. Zwiebach, Oriented open-closed string theory revisited, 
Annals Phys. {\bf 267} (1998) 193; hep-th/9705241.

\end{thebibliography}
\end{document}